\documentclass[aos]{imsart}
\RequirePackage{amsthm,amsmath,amsfonts,amssymb,amscd}
\RequirePackage[numbers]{natbib}
\usepackage{pdfsync}
\usepackage[utf8]{inputenc}
\usepackage[english]{babel}
\usepackage{graphicx,epsfig}
\usepackage{booktabs}
\usepackage{makeidx}
\makeindex

\usepackage[table]{xcolor}
\startlocaldefs

\def\R{\mathbb{R}}

\def\N{\mathbb{N}}
\def\P{\mathbb{P}}
\def\E{\mathbb{E}}

\def\R{\mathbb{R}}

\def\Z{\mathbb{Z}}

\def\cov{\mbox{Cov}\,}

\def\1{\mbox{\hspace{.2em}I\hspace{-.6em}1}} 

\theoremstyle{plain}
\theoremstyle{remark}
\DeclareMathOperator{\argmin}{argmin}
\DeclareMathOperator{\argmax}{argmax}
\newtheorem{theo}{Theorem}[section]
\newtheorem{lemma}{Lemma}

\newtheorem{proposition}{Proposition}
\numberwithin{equation}{section}
\newtheorem{corollary}{Corollary}
\newtheorem{remark}{Remark}

\newtheorem{definition}{Definition}

\def\limiten{\renewcommand{\arraystretch}{0.5}
\begin{array}[t]{c}
\stackrel{}{\longrightarrow} \\
{\scriptstyle n \rightarrow\infty}
\end{array}\renewcommand{\arraystretch}{1}}

\def\simn{\renewcommand{\arraystretch}{0.5}
\begin{array}[t]{c}
\stackrel{}{\sim} \\
{\scriptstyle n \rightarrow\infty}
\end{array}\renewcommand{\arraystretch}{1}}

\def\limiteproban{\renewcommand{\arraystretch}{0.5}
\begin{array}[t]{c}
\stackrel{{\mathcal P}}{\longrightarrow} \\
{\scriptstyle n \rightarrow\infty}
\end{array}\renewcommand{\arraystretch}{1}}

\def\limiteloin{\renewcommand{\arraystretch}{0.5}
\begin{array}[t]{c}
\stackrel{{\mathcal D}}{\longrightarrow} \\
{\scriptstyle n \rightarrow\infty}
\end{array}\renewcommand{\arraystretch}{1}}

\def\egaleloi{\renewcommand{\arraystretch}{0.5}
\begin{array}[t]{c}
\stackrel{{\mathcal D}}{\sim} \\
{\scriptstyle }
\end{array}\renewcommand{\arraystretch}{1}}

\def\limiteasn{\renewcommand{\arraystretch}{0.5}
\begin{array}[t]{c}\stackrel{a.s.}{\longrightarrow} \\
{\scriptstyle
n\rightarrow+\infty}\end{array}\renewcommand{\arraystretch}{1}}

\def\limitesimn{\renewcommand{\arraystretch}{0.5}
\begin{array}[t]{c}\stackrel{}{\sim} \\
{\scriptstyle
n\rightarrow \infty}\end{array}\renewcommand{\arraystretch}{1}}

\endlocaldefs
\begin{document}
\begin{frontmatter}
\title{Efficient and consistent data-driven model selection for time series}
\runtitle{Data-driven model selection for time series}
\begin{aug}
\author[A]{\fnms{Jean-Marc} \snm{Bardet} \ead[label=e1]{bardet@univ-paris1.fr}},
\author[A]{\fnms{Kamila} \snm{Kare} \ead[label=e2]{kamilakare@gmail.com}}
\and
\author[C]{\fnms{William} \snm{Kengne} \ead[label=e3]{william.kengne@u-cergy.fr}}
\address[A]{University Paris 1 Panthéon-Sorbonne, SAMM, France, \printead{e1}, \printead{e2}}
\address[C]{CY Cergy Paris Université, THEMA, France, \printead{e3}}
\end{aug}
\begin{abstract}
~~This paper studies the model selection problem in a large class of causal time series models, which includes both the ARMA or AR($\infty$) processes, as well as the GARCH or ARCH($\infty$), APARCH, ARMA-GARCH and many others processes. We first study the asymptotic behavior of the ideal penalty that minimizes the risk induced by a quasi-likelihood estimation among a finite family of models containing the true model. Then, we provide general conditions on the penalty term for obtaining the consistency and efficiency properties. We notably prove that consistent model selection criteria outperform classical AIC criterion in terms of efficiency. Finally, we derive from a Bayesian approach the usual BIC criterion, and by keeping all the second order terms of the Laplace approximation, a data-driven criterion denoted KC'. Monte-Carlo experiments exhibit the obtained asymptotic results and show that KC' criterion does better than the AIC and BIC ones in terms of consistency and efficiency.
\end{abstract}
\begin{keyword}[class=MSC2010]
\kwd[Primary ]{62G05}
\kwd{
62G20}
\kwd[; secondary ]{62M05}
\end{keyword}
\end{frontmatter}
\tableofcontents

\section{Introduction}
Model selection is one of the fundamental tasks in Statistics and Data Science. It aims at providing a model (or an algorithm) that is the best, following a criterion, to represent observed data. Two leading model selection procedures have received a lot of attention in the literature. On one hand, the resampling methods such as hold out or more generally $V$-fold cross-validation are widely used in the machine learning community. On the other hand, the methods based on the minimization of a penalized risk are also now very popular in the community of applied or theoretical statisticians. They are certainly more appropriate to be applied to time series since they take into account the dependence between data. This will be our choice in this paper.
\\
~\\
The main challenging task when designing a penalized based criterion is the calibration of the penalty. This is mainly dependent on the goal one would like the final criterion achieves. For instance, the objective could be the \textit{consistency}, the \textit{efficiency} or  the \textit{adaptive nature in the minimax sense} to name a few.
\\
~\\
The consistency property aims at identifying the data generating process with high probability.  Hence, it requires the assumption whereby there exists a true model in the set of competitive models and the goal is to select this with probability approaches one as the sample size tends to infinity. Although the consistency is a convincing mathematical property, this asymptotic property is not always the most interesting when switching to a practical implementation. Indeed, the true underlying process is generally unknown and trying to identify the true model for any data is quite ambitious. It is often more plausible to assume that the true data generating process is infinite-dimensional, and that one tries to identify a "good" finite-dimensional model based on the data (\cite{Hurvich1989}). Therefore, it is common in this framework to let the dimension of the competitive models to depend on the number of observations in order to obtain a better approximation and to reduce the  prediction's risk. Hence, the model selection is said to be efficient when its risk is asymptotically equivalent to the risk of the \textit{oracle}. \\
~\\
In this work, we are interested by providing efficient and consistent penalized data-driven criteria for affine causal time series, which are defined by:  
\\
~\\
\textbf{Class} $\mathcal{AC}(M,f):$ {\it A process $X=(X_t)_{t\in \Z}$ belongs to $\mathcal{AC}(M,f)$ if it satisfies:
\begin{equation}
 X_t=M\big((X_{t-i})_{i\in \mathbb{N}^*}\big)\, \xi_t+f\big((X_{t-i})_{i\in \mathbb{N}^*}\big) \;\; \mbox{for any}~ t\in \Z.
\label{eq:serie2}
\end{equation}
where $(\xi_t)_{t\in T}$ is a sequence  of zero-mean independent and identically distributed random vectors (i.i.d.r.v) satisfying $\mathbb{E}(|\xi_0|^r)<\infty $ with $r\geq 1$  and $M$, $f$ : $\R^\infty \to \R$ are two measurable functions, where $R^\infty$ is the set of numeric sequence with finite number of non-zero terms. }\\
~\\
 \noindent For instance,
\begin{itemize}
	\item if $M\big((X_{t-i})_{i\in \mathbb{N}^*}\big)=\sigma$ and $f\big((X_{t-i})_{i\in \mathbb{N}^*}\big)=\phi_1X_{t-1}+\cdots+\phi_pX_{t-p}$,
	then $(X_t)_{t \in \Z}$ is an AR$(p)$ process;
	\item if $M\big((X_{t-i})_{i\in \mathbb{N}^*}\big)=\sqrt{a_0+a_1X^2_{t-1}+\cdots+a_pX^2_{t-p}}$ and $f\big((X_{t-i})_{i\in \mathbb{N}^*}\big)=0$, then $(X_t)_{t \in \Z}$ is an ARCH$(p)$ process.
\end{itemize}

\noindent Note that numerous classical  time series models such as  ARMA($p,q$), GARCH($p,q$), ARMA($p,q$)-GARCH($p,q$) (see \cite{ding} and \cite{ling})
or APARCH$(\delta,p, q)$ processes (see \cite{ding}) belong to $\mathcal{AC}(M,f)$.
 
\noindent The study of these causal affine time series  more often requires the classical regularity conditions on the functions $M$ and $f$ that are not really restrictive and remain valid for many time series.  

\medskip

\noindent 
We will consider the semi-parametric class of models $\mathcal{AC}(M_\theta,f_\theta)$ where  $\theta \in \Theta$ (a compact subset of $\R^d$, $d \in \N$), where $(f_{\theta})_{\theta \in \Theta}$ and $(M_{\theta})_{\theta \in \Theta}$ are two families of functions such as for $\theta \in \Theta$, $f_{\theta}:\R^\infty \to \R$ and  $M_{\theta}:\R^\infty \to [0,\infty)$ are known and the distribution of $\xi_0$ is unknown. A finite family of models ${\cal M}$ (for instance, the class of AR(p) processes where $0\leq p \leq p_{\max}$) will be considered, where a model $m\in {\cal M}$ corresponds to a linear subspace of $\R^d$. A trajectory $(X_1,\cdots,X_n)$ generated from the class $\mathcal{AC}(M_{\theta^*},f_{\theta^*})$ with the "true" model $m^* \in \mathcal{M}$ is supposed to be observed (see Section \ref{sec:2}).\\
~\\
\noindent There already exist several important contributions devoted to the model selection for time series; we
refer to the book of \cite{mcQ} and the references therein for an overview on this topic. Also, the time series model selection literature is very extensive and still growing; we refer to the monograph of \cite{Rao2001}, which provided an excellent summary of existing model selection procedure, including the case of time series models as well as the recent review paper of \cite{Ding2018}. The asymptotically efficient selection property has already been tackled in case of linear process type $AR(\infty)$ by \cite{Shibata}, \cite{Shao1997}, \cite{Karagrigoriou1997}, \cite{Ing2005}, \cite{Ing2012},  and recently by \cite{hsu2019model}.  \\
~\\
The study of a consistent model selection in this class of affine causal processes has been also considered by \cite{bar2019} and \cite{Kengne2021}. As in these papers, we consider here a risk built from the Gaussian conditional log-Likelihood, which is naturally deduced for all causal affine models $\mathcal{AC}(M_\theta,f_\theta)$, and consider a model selection criterion as a penalized Gaussian Quasi-Maximum conditional log-Likelihood (see for instance \cite{francq2019} or \cite{barW}). This allows us:
\begin{enumerate}
\item To study the asymptotic behavior of an ideal penalty that is defined as providing a minimization of the risk;
\item To determine the conditions for obtaining (or not) the asymptotic consistency of a criterion, {\it i.e.} that it allows asymptotically to select the true model;
\item To determine the conditions for approaching (or not) in $1/n$ or even in $o(1/n)$ the minimal risk, thus to obtain an asymptotic efficiency;
\item To determine from a Bayesian approach the BIC criterion as well as a second data-driven criterion called KC', which is obtained by keeping all the second order terms in the Laplace approximation and to prove that they verify the properties of asymptotic consistency and efficiency in $o(1/n)$.
\end{enumerate}
In the end, we show that in the chosen framework the BIC and KC' criteria offer all the advantages with respect to the classical AIC criterion, which allows neither the asymptotic consistency nor the same efficiency. Numerical simulations confirm these results and also show that the new data-driven KC' criterion clearly outperforms the BIC criterion both in terms of consistency and efficiency.\\ 
~\\
\noindent The paper is organized as follows.
The model selection framework based on Gaussian Likelihood risk is described in Section
\ref{sec:2}. In Section \ref{TLCm}, the precise assumptions are stated and they lead to new asymptotic results satisfied by the Quasi-Maximum Likelihood Estimator. The asymptotic behavior of the ideal penalty is studied in Section \ref{sec:resul} as well as some conditions for obtaining the asymptotic efficiency or consistency of a criterion. In Section \ref{sec:bay}, the usual BIC criterion  as well as the data-driven criterion KC' are studied. Finally, examples are detailed in Section \ref{Exam}, numerical results are presented in Section \ref{sec:simu} and Section \ref{sec:proof} contains the proofs.

\section{Model selection framework}\label{sec:2}

\subsection{Finite family ${\cal M}$ of parametric affine causal models} \label{finite_family_AC}
Assume a trajectory $(X_1, \ldots, X_n)$ is observed from a causal stationary solution of (\ref{eq:serie2}) where $M$ and $f$ are two known functions depending on an unknown finite dimensional vector of parameters $\theta^*$.  \\
~\\
Now consider a finite family $\cal M$ of models belonging to parametric affine causal models. In Proposition 1 of \cite{bar2019}, due to the linearity of such models and because ${\cal M}$ is a finite family, it was established that it is always possible to find a dimension $d \in \N^*$ and a unique couple of known functions $(M_{\theta},f_{\theta})$ with $\theta \in \R^d$ such as in such a way that any model $m \in {\cal M}$ belongs to the class $\mathcal{AC}(M_{\theta},f_{\theta})$. More precisely, there is a one-to-one correspondence between each model $m \in {\cal M}$ and a linear subspace $\Theta_m \subset \R^d$ and $\dim(\Theta_m)=|m|$ the number of unknown parameters of the model $m$. As a consequence, if we denote $m^*$ the "true" model corresponding to $\mathcal{AC}(M_{\theta^*},f_{\theta^*})$, we will say:
\begin{itemize}
\item if $m\in {\cal M}$ is such that $\Theta_{m^*} \subset \Theta_m$ and $\Theta_{m^*} \neq \Theta_m$ (also denoted $m^*\subset m$ and $m^*\neq m$), this is an overfitting's case; 
\item if $m\in {\cal M}$ is such that $\Theta_{m^*} \not \subset \Theta_m$ (also denoted $m^*\not \subset m$) , this is a misspecified case. 
\end{itemize}
For example, if $m^*$ corresponds to a AR$(2)$ process and if ${\cal M}$ contains AR$(p_{\max})$ processes and ARCH$(q_{\max})$, we have $d=1+p_{\max}+q_{\max}$ and for $\theta=(\theta_i)_{0\leq i \leq d}$, 
$$
f_\theta((X_{t-k})_{k\geq 1})=\sum_{i=1}^{p_{\max}} \theta_i \, X_{t-i}\qquad \mbox{and}\qquad M_\theta((X_{t-k})_{k\geq 1})=\Big (\theta_0+\sum_{i=p_{\max}+1}^{p_{\max}+q_{\max}} \theta_i \, X^2_{t-i} \Big )^{1/2}.
$$
Then $\Theta_{m^*}=\big \{(\theta_0,\theta_1,\theta_2,0,\ldots,0),\, (\theta_0,\theta_1,\theta_2) \in \R^3\big \}$, an AR$(4)$ process implies an overfitting, while an AR$(1)$ or an ARCH$(2)$ process implies a misspecified case. \\
In the sequel, we will always assume that 
$$
m^* \in {\cal M}.
$$
This true model $m^*$ is supposed to be unknown. After observing the trajectory $(X_1,\ldots, X_n)$, our goal is to find the most probable model (see Section \ref{sec:bay}) among the finite family $\cal M$ or a "best" model that forecasts with a minimum risk. Here we have chosen a risk that is derived from a Quasi-Maximum Likelihood contrast, which is presented below.


\subsection{Maximum Likelihood Estimation} \label{sub:1}

For each $\theta \in \Theta$, we will begin by defining its risk by:
\begin{multline}\label{R}
R(\theta):=\P \gamma(\theta)=\E[ \gamma(\theta, X_1)] \\
\mbox{with} \quad \gamma(\theta, X_{t}):= \frac{(X_t - f_{\theta}^t)^2}{H_{\theta}^t} + \log(H_{\theta}^t)\quad
\mbox{and}\quad
\displaystyle \left \{
\begin{array}{lcl}
{f}_{\theta}^t&:=&f_{\theta}\big ( (X_{t-k})_{k\geq 1}\big )\\
{M}_{\theta}^t&:=&M_{\theta}\big ( (X_{t-k})_{k\geq 1}\big ) \\
{H}_{\theta}^t&:=&\big (M_{\theta}^t\big )^2
\end{array} \right . .
\end{multline}
By referring to \cite{massart2007} or \cite{francq2019}, the contrast $\gamma(\theta,.)$ is $-2$ times the Gaussian conditional log-density of $X_t$. Moreover, the Gaussian Maximum Likelihood Estimator (MLE) is derived from the conditional (with respect to the filtration $\sigma \big \{(X_{t})_{t\leq 0} \big \}$) log-likelihood of $(X_1,\ldots,X_n)$ when $(\xi_t)$
is supposed to be a Gaussian standard white noise. We deduce that this conditional log-likelihood (up to an additional constant) $L_n$ is defined for a parameter $\theta$ by:
\begin{equation}
L_n(\theta):=-\frac{1}{2}\sum_{t=1}^n \gamma(\theta, X_t).
\label{eq:eq1}
\end{equation}
As it has been proved in \cite{barW}, under a classical identifiability assumption, the risk function $R$ achieves its unique minimum at the "true"' parameter $\theta^*$  over any parameter set $\Theta$, when $\theta^* \in \Theta$, {\it i.e.}
\begin{equation}\label{theta*}
\theta^*= \underset{\theta \in \Theta}{\argmin} \, R(\theta).
\end{equation}
Therefore $\theta^*$ is considered as an ideal predictor for the model selection procedure and serves as a benchmark to compare predictors. Given a model $m \in {\cal M}$ and $\Theta_m$ its parameter space that does not necessarily contains $\theta^*$, let us define
\begin{equation}\label{thetam*}
\theta_m^*= \underset{\theta \in \Theta_m}{\argmin} \; R(\theta). 
\end{equation}
As a consequence, we have:
$$
\theta_{m^*}^*=\theta^*\quad\mbox{and if $m^*\subset m$,}\quad \theta_{m}^*=\theta^*.
$$
Besides of minimizing the risk $R$, we also consider the minimization of its associated loss function, which is defined as 
\begin{equation}
\ell (\theta, \theta^*):=R(\theta)-R(\theta^*) \ge 0.
\label{eq:risk}
\end{equation}
This is a well-known measure of separation between the candidate model generated by $\theta$ and the true one indexed by $\theta^*$.\\~
\\
Let set by $\gamma_n$ the associated empirical risk defined by
$$
\gamma_n(\theta):=\P_n \gamma(\theta,.)= \frac{1}{n} \sum_{t=1}^{n} \gamma(\theta, X_t)=- \frac{2}{n} L_n(\theta),
$$
so that maximizing the log-likelihood is equivalent to minimize the empirical criterion $\gamma_n.$

\subsection{Quasi-Maximum Likelihood Estimation}
Since the  white noise is not necessarily a Gaussian one and since the log-likelihood (and then the empirical risk) $L_n(\theta)$ depends on $(X_t)_{t\le 0}$ that are unknown, a quasi-log-likelihood $\widehat L_n(\theta)$ can be used as an approximation of the log-likelihood. It It consists of replacing $\gamma(\theta, X_t)$ by an approximation $\widehat{\gamma}(\theta, X_t)$ and those statistics are defined for all $\theta \in \Theta$ by
\begin{multline}
\widehat{L}_n(\theta):=-\frac{1}{2}\, \sum_{t=1}^n \widehat{\gamma}(\theta, X_t)\quad  \\
\textnormal{with} ~ \widehat{\gamma}(\theta, X_t):=\frac{(X_t - \widehat{f}_{\theta}^t)^2}{\widehat{H}_{\theta}^t} + \log(\widehat{H}_{\theta}^t) \quad
\mbox{and}~
\displaystyle \left \{
\begin{array}{lcl}
\widehat{f}_{\theta}^t&:=&f_{\theta}(X_{t-1},X_{t-2},\cdots,X_1,u)\\
\widehat{M}_{\theta}^t&:=&M_{\theta}(X_{t-1},X_{t-2},\cdots,X_1,u)  \\
\widehat{H}_{\theta}^t&:=&(\widehat{M}_{\theta}^t)^2
\end{array} \right .
\end{multline}
for any deterministic sequence $u=(u_n)_{n\in \N}$ with finitely many non-zero values (we will use $u =0$ without loss of generality).

\noindent In addition the computable empirical risk is then:
$$
\widehat{\gamma}_n(\theta)=\P_n \widehat{\gamma}(\theta, .)=\frac 1 n \,\sum_{t=1}^n \widehat{\gamma}(\theta, X_t)=- \frac{2}{n}\widehat{L}_n(\theta). 
$$ 
\noindent Finally, for each specific model $m\in {\cal M}_n$, we define a Gaussian Quasi-Maximum Likelihood Estimator (QMLE) $\widehat{\theta}_m$ as
\begin{equation}
\widehat{\theta}_m \in \underset{\theta \in \Theta_m}{\argmax} \;\widehat{L}_n(\theta)= \underset{\theta \in \Theta_m}{\argmin} \;\widehat{\gamma}_n(\theta).
\label{eq:qmle}
\end{equation}
The estimator $\widehat{\theta}_m$ is commonly called the \textit{Empirical Risk Minimizer} (ERM).  

\subsection{The penalization procedure}
For $m \in \mathcal{M}$, the ERM provides an estimator  in $\Theta_m$. The goal is to come up with a model that minimizes the excess loss over $\mathcal{M}$
\begin{equation}
\underset{m \in \mathcal{M}}{\argmin}~ \ell(\widehat{\theta}_{m},\theta^*).
\label{eq:orac}
\end{equation}
This model is unknown since \eqref{eq:orac} depends on $\theta^*$ and the distribution $P_{(X_1,\ldots,X_n)}$ that are unknown. 


~\\
A classical way to solve \eqref{eq:orac} problem is to design for every $m \in \mathcal{M}$ an estimator of $R(\widehat{\theta}_{m})$ and we naturally choose $\widehat{\gamma}_n(\widehat{\theta}_{m})$.
First, it is well known that the empirical criterion $\widehat{\gamma}_n(\widehat{\theta}_{m})$ is an optimistic version of $R(\widehat{\theta}_{m})$ and decreases with the dimension of the model. Therefore, it is common to add a penalty term to counteract this bias.\\
As a consequence, define a function pen: $m\in \mathcal{M} \mapsto \textnormal{pen}(m) \in \R^+$, which is called the penalty function and is possibly data-dependent. We will only require that $\textnormal{pen}(m_1) \leq \textnormal{pen}(m_2)$ when $m_1 \subset m_2$. Then define the \textit{penalized contrast} and the model selected by it:
\begin{equation}
\widehat{m}_{\textnormal{pen}}= \underset{m \in \mathcal{M}}{\argmin} \; \big \{\widehat{C}_{\textnormal{pen}}(m) \big \} \qquad \textnormal{with} \qquad \widehat{C}_{\textnormal{pen}}(m):=  \widehat{\gamma}_n\big(\widehat{\theta}_m\big)+\textnormal{pen}(m).
\label{eq:crit}
\end{equation}
In order to achieve \eqref{eq:orac}, the {\it ideal penalty} to consider in \eqref{eq:crit}  is 
\begin{equation}
\textnormal{pen}_{id}(m)=R(\widehat{\theta}_{m})- \widehat{\gamma}_n(\widehat{\theta}_m). 
\label{eq:penid}
\end{equation}
Using this definition, we obtain an "ideal" model defined by:
\begin{equation}
\widehat m_{id}:\in \underset{m \in \mathcal{M}}{\argmin} \,\big \{\ell(\widehat{\theta}_{m},\theta^*)\big \}=\underset{m \in \mathcal{M}}{\argmin} \, \big \{ R(\widehat{\theta}_{m})\big \}=\underset{m \in \mathcal{M}} {\argmin}\,\big \{\widehat{C}_{\textnormal{pen}_{id}}(m) \big \} . 
\label{mid}
\end{equation}
However, the function $R$ is unknown except for very few particular and parametric cases and therefore $\textnormal{pen}_{id}$ cannot generally be used directly. Therefore the question is how to choose the penalty in \eqref{eq:crit} so that  $\widehat{m}_{\textnormal{pen}}$ mimics the \textit{oracle} $\widehat m_{id}$. Hence, we would like our final estimator $\widehat{\theta}_{\widehat{m}_{\textnormal{pen}}}$ to behave asymptotically  like the oracle. That is to satisfy:
 \begin{equation}
\P \Big ( \ell(\widehat{\theta}_{\widehat{m}_{\textnormal{pen}}},\theta^* )\le \underset{m \in \mathcal{M}}{\min} \big\{\ell(\widehat{\theta}_{m},\theta^* )\big\}+\frac{C}{n} \Big ) \limiten 1
\label{ao}
\end{equation}
and/or for any $n \ge n_0$
\begin{equation}
\E\big[	\ell(\widehat{\theta}_{\widehat{m}_{\textnormal{pen}}},\theta^* )\big] \le \underset{m \in \mathcal{M}}{\min} \Big\{\E\big[\ell(\widehat{\theta}_{m},\theta^*)\big]\Big\} +\frac {C}{n}.
\label{ao1}
\end{equation}
The aim of this paper is to find a good choice of $\mbox{pen}(m)$ in order to obtain  the {\it asymptotic efficiency} \eqref{ao1} or \eqref{ao}.
\section{Asymptotic behavior of the QMLE} \label{TLCm}
Before considering the problem of model selection, we establish a central limit theorem satisfied by $\widehat{\theta}_{m}$ for any model $m\in {\cal M}$, {\it i.e.} as well if $m$ is an overfitted or a misspecified model. Before this, some notations and assumptions have to be precised.
\subsection{Notations and main assumptions}
In the sequel, we will consider a subset $\Theta$  of $ \mathbb{R}^{d}$ which is compact. We will use the following norms:
\begin{itemize}
	\item $\|.\|$  denotes the usual Euclidean norm on $\R^\nu$, with $\nu\geq 1$;
	\item for a matrix $A$, denote  $\|A\|$ the subordinate matrix norm such that  $\|A\|=\underset{v \ne 0 }{\sup} \, \frac{\|A \, v\|}{\|v\|}$;
	\item if $X$ is a $\mathbb{R}^\nu$-random variable and $r\ge 1$, we set $\|X\|_r=\big (\E \big [\|X\|^r\big ]\big)^{1/r} \in [0,\infty]$;
	\item for $\theta \in \Theta \subset \R^d$, if $\Psi_{\theta}:\R^\infty \to E$ where $E=\R^\nu$ or $E$ is a set of square matrix, denote $\|\Psi_{\theta}(\cdot)\|_{\Theta}=\underset {\theta \in \Theta}{\sup} \big \{ \|\Psi_{\theta}(\cdot)\|\big \}$;
	\item for $\theta \in \Theta \subset \R^d$, if $\Psi_{\theta}:\R^\infty \to \R$ is a ${\cal C}^2(\Theta\times R^\infty)$ function, we will denote 
	$$\partial_\theta \Psi_{\theta}(\cdot)=\Big ( \frac {\partial}{\partial \theta_i} \Psi_{\theta}(\cdot)\Big )_{1\leq i \leq d}=\big ( \partial _{\theta_i} \Psi_{\theta}(\cdot)\big )_{1\leq i \leq d}\quad\mbox{and}\quad\partial^2_{\theta^2} \Psi_{\theta}(\cdot)=\Big ( \frac {\partial^2}{\partial \theta_i\partial \theta_j} \Psi_{\theta}(\cdot)\Big )_{1\leq i,j \leq d};
	$$
	\item consider $\Psi_{\theta}:\R^\infty \to \R$ for any $\theta \in \Theta \subset \R^d$. Then, we define the assumption: \\
~\\
\textbf{Assumption A}$(\Psi_{\theta},\Theta)$: $\|\Psi_{\theta}(0)\|_\Theta < \infty$ and there exists a sequence of non-negative real numbers $\big(\alpha_k(\Psi_{\theta},\Theta)\big)_{k\ge 1}$ such that $\sum_{k=1}^{\infty}\alpha_k(\Psi_{\theta},\Theta)< \infty$ and satisfying:
	\[
	\hspace*{-3mm}
	\|\Psi_{\theta}(x)-\Psi_{\theta}(y)\|_\Theta \le \sum_{k=1}^{\infty}\alpha_k(\Psi_{\theta},\Theta)|x_k-y_k|\quad\mbox{for all $x,y \in \mathbb{R}^\mathbb{\infty}$}.
	\]
\end{itemize} 	
Several assumptions on the AC class will be considered thereafter:

\medskip

\noindent \textbf{Assumption A0}: {\it  The process $X \in \mathcal{AC}(M_{\theta^*},f_{\theta^*})$ where $\theta^*\in \Theta$ is defined in \eqref{eq:serie2} where:
\begin{itemize}
\item the white noise $(\xi_t)_t$ is such as $\|\xi_0\|_r<\infty$ with $8<r$;
\item for any $x \in \R^\infty$, the functions $\theta \to M_\theta$ and $\theta \to f_\theta$ are ${\cal C}^2(\Theta)$ functions:
\item $\Theta \in \R^d$ is a compact set such as
\begin{multline}\label{Theta}
\Theta \subset \Big \{\theta  \in \R^d,~ A(f_{\theta},\{\theta\})\; \textnormal{and}\; A(M_{\theta},\{\theta\})\; \textnormal{hold with} \\
\sum_{k=1}^{\infty} \alpha_k(f_{\theta},\{\theta\}) +\|\xi_0\|_r \, \sum_{k=1}^{\infty} \alpha_k(M_{\theta},\{\theta\}) < 1 \Big \}.
\end{multline}
\end{itemize}}
\noindent Under this assumption, \cite{dou} showed that  there exists a stationary causal ({\it i.e.} $X_t$ is depending only on $(X_{t-k})_{k\in \N}$ for any $t\in \Z$) and ergodic solution of (\ref{eq:serie2}) with $r$-order moment for any $\theta \in \Theta$.\\
~\\
Now the assumption {\bf A0} holds. We will also add several assumptions required for insuring the strong consistency and the asymptotic normality of the QMLE:\\
~\\
The first following classical assumption ensures the identifiability of $\theta^*$.

\medskip

\noindent \textbf{Assumption A1}: {\it  For all $\theta, \, \theta' \in \Theta$,
	$(f_{\theta}^0=f_{\theta'}^0 \; and\;  M_{\theta}^0=M_{\theta'}^0 ) ~ a.s. \implies \theta=\theta'. $ } 
\begin{remark}\label{Ident}
Even if this assumption is a classic one in an M-estimation framework, it is important to remark that it does not cover all the cases of model selection of usual causal time series. Indeed, in the case of the family of ARMA processes, it is well known that a model is unique when both the polynomials $P_0$ and $Q_0$ of AR and MA parts are coprime. Hence, for instance, if the true model is an ARMA$(p_0,q_0)$ process, any ARMA$(p_0+1,q_0+1)$ representation with respective polynomials $P(X)=P_0(X)(X-r)$ and $Q(X)=Q_0(X)(X-r)$ of AR and MA parts, is identically the same as the true model whatever $r\in \R$. Then,  Assumption {\bf A1} is never satisfied for ARMA processes in case of overfitting. However, by initializing $\theta$ around $0$ in the optimization algorithm, we have noticed from Monte-Carlo experiments that the algorithm always converges to $\theta^*$ and not other solution.
\end{remark}

\noindent Next, the following Assumption ensures the invertibility of the asymptotic covariance matrix $G$ and $F$ (see below) that is necessary to prove the asymptotic normality of the QMLE (see for instance \cite{barW}).

\medskip

\noindent \textbf{Assumption A2}: {\it $<\alpha, \partial_\theta f^0_{\theta}>=0~\!\implies\! ~\alpha=0~a.s.~$ or $~<\alpha, \partial_\theta M^0_{\theta}>=0~\!\implies \! ~\alpha=0~a.s.$}

\medskip

%
%
 
\noindent The definition of the computable empirical risk and requires that its denominators do not vanish. Hence, we are going to assume throughout this paper that the lower bound of $H_{\theta}(\cdot) = \big(M_{\theta}(\cdot) \big)^2$  is strictly positive:

\medskip

\noindent \textbf{Assumption A3}: {\it $\exists \underline{h}>0$ such that $\underset{\theta \in \Theta}{\inf}(H_{\theta}(x)) \ge \underline{h}$ for all $x\in \mathbb{R}^{\infty}$.}

\medskip

\noindent The following assumption is a technical classical condition (see \cite{lvliu}).

\medskip

\noindent \textbf{Assumption A4}: {\it For every $m \in \mathcal{M}$, if $(\overline{\theta}_{m,n})$ is a sequence of $\Theta_m$ satisfying $\overline{\theta}_{m,n} \limiteasn \theta^*$, then  
\begin{equation} \label{eqA4} 
  \underset{n \rightarrow \infty}{\limsup}~ \Big\{\E\Big[\Big ( \Big \| \frac 1 n \, \big ( \partial^2_{\theta_i \theta_j} L_n(\overline \theta_m) \big )_{i,j\in m}\Big )^{-1}\Big \|^8\Big]\Big\}  < \infty. 
 \end{equation} }
\begin{remark} Note that under assumption {\bf A0},  if $\overline{\theta}_{m,n} \limiteasn \theta_m^*$ then 
$$
 \Big \| \Big ( \frac 1 n \, \big ( \partial^2_{\theta_i \theta_j} L_n(\overline \theta_m) \big )_{i,j\in m}\Big )^{-1}\Big \|^8 \limiteasn  \Big \| \Big ( \big (-\frac 1 2 \,  \partial^2_{\theta_i \theta_j} \gamma(\theta^*_m) \big )_{i,j\in m}\Big )^{-1}\Big \|^8 
$$ 
Thus, from the  Egorov's Theorem, we can find an event $\widetilde{\Omega}$ with sufficiently large probability such that the relation (\ref{eqA4}) in the assumption {\bf A4} holds if the expectation is taken on the event $\widetilde{\Omega}$.
 For the particular case of the linear processes, the assumption {\bf A4} holds true under a mild condition on the distribution of $X$,
 see for instance \cite{papa} and \cite{Findley2002}. 
\end{remark}

\noindent  Finally, the decrease rates of $(\alpha_j (f_\theta,\Theta))_j$, $(\alpha_j (M_\theta,\Theta))_j$, $(\alpha_j (\partial_{\theta} f_\theta,\Theta))_j$ and $(\alpha_j (\partial_{\theta} M_\theta,\Theta))_j$ have to be fast enough for insuring the strong consistency and the asymptotic normality of the QMLE:
\medskip

\noindent \textbf{Assumption A5}: {\it Conditions {\bf A}$(f_{\theta},\Theta)$, {\bf A}$(M_{\theta},\Theta)$, {\bf A}$(\partial_{\theta}f_{\theta},\Theta)$, {\bf A}$(\partial_{\theta} M_{\theta},\Theta)$, {\bf A}$(\partial^2_{\theta^2}f_{\theta},\Theta)$ and  {\bf A}$(\partial^2_{\theta^2} M_{\theta},\Theta)$ hold with 
$$\alpha_j (f_\theta,\Theta)+ \alpha_j (M_\theta,\Theta) + \alpha_j (\partial_{\theta} f_\theta,\Theta)+ \alpha_j (\partial_{\theta} M_\theta,\Theta) =O(j^{-\delta})\quad \mbox{where}\quad \delta>7/2.
$$}
~\\
\noindent Note that Assumption {\bf A5} does not allow to consider long-range dependent processes, but usual short memory causal time series satisfy this assumption.  
\subsection{New asymptotic results satisfied by $\widehat \theta_m$}
The asymptotic normality of $\widehat \theta_m$ has been already established in \cite{barW} when $m=m^*$ and in \cite{bar2019} when $m^*\subset m$ (overfitting). This property can also be extended in the case of misspecified model, {\it i.e.} when $m^* \not \subset m$.   \\
~\\
First, the following corollary is a particular case of a more general result, Proposition \ref{theo0}, which is stated in Section \ref{sec:proof} devoted to the proofs. \\
To begin with, and from Assumption {\bf A2} and {\bf A5}, we can define the definite positive matrix 
\begin{eqnarray}\label{matrixG}
G(\theta)&:=&\frac 1 4 \, \Big (\sum_{t\in \Z} \cov \big ( \partial_{\theta_i} \gamma(\theta,X_0)\, , \, \partial_{\theta_j} \gamma(\theta,X_t)\big )\Big )_{1 \leq i,j\leq d} \\
\label{matrixF} F(\theta)&:=&-\frac 1 2 \, \Big ( \E \Big[\partial^2_{\theta_i \,\theta_j} \gamma(\theta,X_0)\Big] \Big )_{1 \leq i,j\leq d}.
\end{eqnarray}
Now, for any $m\in {\cal M}$ and $\theta \in \Theta$, denote:
\begin{eqnarray}\label{GG}
\qquad G_m(\theta)&:=&\frac 1 4 \, \Big (\sum_{t\in \Z} \cov \big ( \partial_{\theta_i} \gamma(\theta,X_0)\, , \, \partial_{\theta_j} \gamma(\theta,X_t)\big )\Big )_{i,j\in m} \\
\nonumber \Longrightarrow && G_m(\theta^*) = \frac 1 4 \, \Big ( \cov \big ( \partial_{\theta_i} \gamma(\theta^*,X_0)\, , \, \partial_{\theta_j} \gamma(\theta^*,X_0)\big )\Big )_{i,j\in m}\quad \mbox{if $m^*\subset m$}  \\
\label{FF} \qquad F_m(\theta)&:=&-\frac 1 2 \, \Big ( \E \Big[\partial^2_{\theta_i \,\theta_j} \gamma(\theta,X_0)\Big] \Big )_{i,j\in m}.
\end{eqnarray}
\begin{corollary}\label{corol}
Let $m\in {\cal M}$ and suppose that Assumptions {\bf A0}-{\bf A5} hold. 
Then, with $\theta_m^*$ defined in \eqref{thetam*},
\begin{equation}\label{tlcthetam}
\frac 1 {\sqrt{n}} \, \Big (  \partial _{\theta_j} L_n(\theta_m^*)  \Big )_{j\in m}  \limiteloin {\cal N} \big ( 0 \, , \, G_m(\theta^*_m) \big ).
\end{equation}
\end{corollary}
\noindent Using mainly this new result, we also obtain:
\begin{theo} \label{theomiss}
Under Assumptions {\bf A0}-{\bf A5}, for any $m\in {\cal M}$,
\begin{equation}\label{theothetamiss}
\sqrt n \, \big ( (\widehat \theta_m)_i-(\theta^*_m)_i \big )_{i \in m} \limiteloin {\cal N} \Big (0 \, , \,\big (F_m(\theta^*_m) \big )^{-1} G_m(\theta^*_m) \,\big (F_m(\theta^*_m) \big )^{-1}  \Big ),
\end{equation}
with $G_m$ and $F_m$ defined in \eqref{GG} and \eqref{FF}.
\end{theo} 
\noindent Hence, even in the misspecified case, $\widehat \theta_m$ satisfies a central limit theorem. We will use this result several times, in particular to prove that the probability of selecting a misspecified model tends quickly enough towards $0$. Another technical result will also be useful for the sequel:
\begin{proposition}\label{2+}
Under Assumptions {\bf A0}-{\bf A5}, with $8/3<r'\leq (8+r)/6$ and $r'<2(\delta-1)$ where $\delta>7/2$ is given in Assumption {\bf A5} and for any $m\in {\cal M}$, then we have
\begin{equation}\label{ineg2+}
\sup_{n\in \N^*} \Big \| \sqrt n \, \big ( (\widehat \theta_m)_i-(\theta^*_m)_i \big )_{i \in m} \Big \|_{r'} <\infty.
\end{equation}
\end{proposition}
\noindent Note that we also have $\sup_{n\in \N^*} \big \| \sqrt n \, \big ( (\widehat \theta_m)_i-(\theta^*_m)_i \big )_{i \in m} \big \|_{2} <\infty$. This result will be essential for establishing the asymptotic behavior of the expectation of the ideal penalty.
\section{Efficient model selection}\label{sec:resul}

The expectation of the ideal penalty \eqref{eq:penid} has been computed (or asymptotically approximated) in several frameworks (see \cite{mallows}, \cite{akaike}, \cite{schwarz}, \cite{Hurvich1989}, \cite{cavanaugh}; etc) and it is most often proportional to the dimension of the model (denoted $|m|$ in the sequel).
\begin{itemize}
\item the penalty is $2 \, |m| \,\sigma^2 /n$ in the regression setting, leading to the famous Mallows's $C_p$ criterion \cite{mallows};
\item the penalty is $2 \, |m|/n$ in the density estimation framework and others, leading to the famous AIC criterion \cite{akaike};
\item the penalty is $|m| \, \log(n) /n$ in the Bayesian density estimation setting and other frameworks, leading to the famous BIC criterion \cite{schwarz} ;
\item the penalty is $\mbox{Trace}\big(B_n\, A_n^{-1}\big)/n$ where $A_n$ is the opposite of the Hessian matrix of the log-likelihood  and $B_n$ the Fisher Information matrix in a general framework issued from a Bayesian setting \cite{lvliu}.
\end{itemize}

\noindent In order to approximate \eqref{eq:penid} in this framework, let first provide a decomposition of this term in order to facilitate the computation. For any model $m\in {\cal M}$, write 
\begin{equation}\label{eq:decom}
\mbox{pen}_{id}(m): =R(\widehat{\theta}_m)-\widehat{\gamma}_n (\widehat{\theta}_m)= I_1(m)+I_2(m)+I_3(m),
\end{equation}
$$
\mbox{with}\quad \left \{ \begin{array}{lll}I_1(m)&:=&R(\widehat{\theta}_m)-R(\theta^*_m) \\ I_2(m)&:=&\widehat \gamma_n(\theta^*_m)-\widehat \gamma_n(\widehat{\theta}_m) \\
I_3(m)&:=&R(\theta^*_m) -\widehat \gamma_n(\theta^*_m)
\end{array} \right . .
$$
\noindent Next we provide a preliminary result about the asymptotic behavior of the terms $I_1(m)$ and $I_2(m)$. 
Then we obtain:
\begin{lemma}\label{prop1}
Under Assumptions {\bf A0}-{\bf A5}, for any model $m \in {\cal M}$, there exists a probability distribution $U^*(m)$ such that 
\begin{eqnarray}	
\nonumber &1.\quad & n \, I_1(m)=n\, \big ( R(\widehat{\theta}_{m})-R(\theta^*_m) \big ) \limiteloin U^*(m) \\
\label{eq:pr1} && \hspace{1.7cm} \mbox{and}\quad \E \big [ n \, I_1(m)\big ]\limiten \E [U^*(m)]=- \mbox{Trace}\Big( \big (F_m(\theta^*_m)\big )^{-1}  \,G_m(\theta^*_m)\Big ).\qquad \\
\nonumber &2.\quad &n \, I_2(m)=n\, \big ( \widehat \gamma_n(\theta^*_m)-\widehat \gamma_n(\widehat{\theta}_m) \big ) \limiteloin U^*(m) \\
\label{eq:pr2}	 && \hspace{1.7cm} \mbox{and}\quad \E \big [ n \, I_2(m)\big ]\limiten - \mbox{Trace}\Big( \big (F_m(\theta^*_m)\big )^{-1}  \,G_m(\theta^*_m)\Big ).
\end{eqnarray}
\end{lemma}
\noindent The proof of this lemma, as well as all the other proofs, can be found in Section \ref{sec:proof}. This result leads to our first main result:
\begin{proposition}\label{prop0}
Under Assumptions {\bf A0}-{\bf A5}, there exists $N_0 \in \N$ such as for any $n \geq N_0$,
\begin{equation}\label{minpen}
\argmin_{m \in {\cal M}} \E \big [ \ell(\widehat \theta_m,\theta^*)\big ]=m^*.
\end{equation}
\end{proposition}
\noindent Another application of Lemma \ref{prop1} is
devoted to an expansion of the expectation of the ideal penalty defined in \eqref{eq:penid}:
\begin{proposition}\label{prop}
Under  Assumptions {\bf A0}-{\bf A5} and for any $m\in {\cal M}$, there exists a bounded sequence $(v^*_n)_{n\in \N^*}$, not depending on $m$ when $m^* \subset m$, and satisfying
\begin{equation}\label{penid}
\E\big[ \mbox{pen}_{id}(m)\big] 
\limitesimn -\frac 2 {n} \, \mbox{Trace}\Big( \big (F_m(\theta^*_m)\big )^{-1}  \,G_m(\theta^*_m)\Big )\Big )+ \frac {v^*_n} n.
\end{equation}
\end{proposition}
\noindent Note that the Slope Heuristic Procedure, which allows to estimate a so called \textit{minimal penalty} (see \cite{arlot}) consists in evaluating the slope of a linear regression of $\widehat \gamma_n(\widehat{\theta}_m)$ onto $|m|$ for $m^*\subset m$ and this is equivalent to estimating the slope of $\displaystyle \frac 1 {n} \,\mbox{Trace}\big(G_m(\theta^*_m)F_m(\theta^*_m)^{-1} \big)$ onto $|m|$ from \eqref{eq:pr2}. We will see that $\mbox{Trace}\big(G_m(\theta^*_m)F_m(\theta^*_m)^{-1} \big)$ behaves as a linear function of $|m|$ in many cases, which also gives legitimacy to this approach in the case of time series after having obtained it in the case of linear regression. The minimal penalty is then $-2 \times$ the estimated slope and this finally corresponds to an approximation of $\E\big[ \mbox{pen}_{id}(m)\big]$. 
The trace of the matrix mentioned above is easily computable in some cases using the explicit forms of the matrices $F(\theta^*)$, $G(\theta^*)$ in \cite{barW}.   Hence, even if the ideal penalty cannot be explicitely obtained, we can replace it with its expectation, {\it i.e.} this trace of  matrix. Then, define for $m\in {\cal M}$, 
\begin{equation}\label{widepen}
\widetilde {\textnormal{pen}}(m):=-\frac 2 {n} \, \mbox{Trace}\Big( \big (F(\theta^*_m)\big )^{-1}  \,G(\theta^*_m)\Big ).
\end{equation}  
As shown in Section \ref{Exam}, this trace is proportional to the dimension of the model $|m|$ in some cases, but could be more complex functions of $|m|$. In case of Gaussian process, we will also see that it corresponds to the AIC penalty, $2/n$. However, we will see that this penalty does not provide a consistent model selection and contrary to the ideal penalty, does not provide  an optimal efficient model criterion.  \\
Before this, we study the probability of not selecting a misspecified model under a general condition on the penalty that is satisfied for instance by $\widetilde {\textnormal{pen}}$ or by BIC criterion:
\begin{theo}\label{theo:0}
Under Assumptions {\bf A0}-{\bf A5} and if for any $\varepsilon>0$,
\begin{equation}\label{condpen}
n \, \P \big ( \textnormal{pen}(m) \geq \varepsilon \big )\limiten 0 \quad \mbox{for any $m\in {\cal M}$}.
\end{equation} 
Then, 
\begin{equation}\label{theo0pen}
n \, \P \big (m^* \not \subset  \widehat m_{\textnormal{pen}} \big ) \limiten 0.
\end{equation}
\end{theo}
\noindent Theorem \ref{theo:0} says that if the penalty asymptotically decreases to $0$ in probability, then the criterion $\widehat{C}_{\mbox{pen}}$ does not select a misspecified model asymptotically. \\
Now, we state the main results of this paper, which specify the convergence rate of $\mbox{pen}$ to obtain an excess loss close to the minimal one over $\mathcal{M}$:
\begin{theo}\label{theo:1}
Under Assumptions {\bf A0}-{\bf A5}, and if  for any $\varepsilon>0$ there exists $K_\varepsilon>0$ such as 
\begin{equation}\label{eq:penal}
\limsup _{ n \to \infty} \max_{m \in {\cal M}} \, \P \Big ( n \, \textnormal{pen}(m) \geq K_\varepsilon \Big) \leq \varepsilon.
\end{equation}
Then for any $\varepsilon >0$, there exists $M_\varepsilon>0$ and $N_\varepsilon\in \N^*$ such as for any $n\geq N_\varepsilon$,
\begin{equation}\label{bornepen}
\P \Big ( 
\ell(\widehat{\theta}_{\widehat{m}_{\textnormal{pen}}},\theta^* )\le \underset{m \in \mathcal{M}}{\min} \big\{\ell(\widehat{\theta}_{m},\theta^*)\big\} +\frac {M_\varepsilon}{n} \Big ) \geq 1-\varepsilon.
\end{equation}
\end{theo}

\begin{remark}
\noindent Let notice that this asymptotic optimality is quite a bit different from the classical one about {\it asymptotic efficiency}, where both the cardinal of the collection $\mathcal{M}$ and the dimension of competitive models are allowed to depend on $n$. However, this is done in the framework where the parameter $\theta^*$ is infinite-dimensional (see for example \cite{Shibata}, \cite{li1987}, \cite{hsu2019model}).
\end{remark}
\noindent Now, we provide a condition on the penalty allowing to obtain a consistent criterion. Moreover, we also prove that such criterion satisfies a sharper efficiency inequality than \eqref{bornepen}:
\begin{theo}\label{Theo7}
Under Assumptions {\bf A0}-{\bf A5}, if the penalty $\mbox{pen}$ satisfies \eqref{condpen}, and if for any  $m \in {\cal M}$ such as $m^*\subset m$ and $m\neq m^*$,
\begin{equation}\label{Cond7}
n \, \E[e_n(m)] \limiten \infty\qquad \mbox{and}\qquad n \,  \E\big [\big |e_n(m)-\E[e_n(m)]\big | \big ]\limiten 0
\end{equation}
with $e_n(m)=\mbox{pen}(m)-\mbox{pen}(m^*)>0$ since $m^*\subset m$ and $m\neq m^*$. Then we have, 
\begin{equation}\label{consist}
\P \big (\widehat m _{\textnormal{pen}}=m^* \big )\limiten 1.
\end{equation}
Moreover, for any $\varepsilon>0$ and $\eta>0$, there exists $N_{\varepsilon, \eta} \in \N^*$ such as for any $n \geq N_{\varepsilon, \eta}$,
\begin{equation}\label{bornewidepen}
\P \Big (  
\ell(\widehat{\theta}_{\widehat{m}_{\textnormal{pen}}},\theta^* ) \le   \ell(\widehat{\theta}_{m^*},\theta^*) +\frac {\eta}{n} \Big) \geq 1 -\varepsilon,
\end{equation}
and there exists $N_\eta \in \N$ such as for any $n\geq N_\eta$,
\begin{equation}\label{borneopti}
\E \big [ 
\ell(\widehat{\theta}_{\widehat{m}_{\textnormal{pen}}},\theta^* )\big ] \le \min_{m\in {\cal M}} \E \big [  \ell(\widehat{\theta}_{m},\theta^*) \big ] +\frac {\eta}{n}.
\end{equation}
\end{theo}
\noindent The best known criterion satisfying the conditions of this theorem and in particular \eqref{Cond7} is certainly the BIC criterion for which $\textnormal{pen}(m)=\frac {\log n}{n} \,|m|$ and therefore $e_n(m)=\frac {\log n}{n} \,\big (|m|-|m^*| \big)$. This is also such a case for Hannan-Quinn criterion ($\textnormal{pen}(m)=\frac {\log (\log n)}{n} \,|m|$, see \cite{HannanQuinn}) or if $\textnormal{pen}(m)=\frac {\sqrt n}{n} \,|m|$ as we used it in \cite{bar2019}. Note also that both the consistent data-driven criteria mentioned in the next section (see \eqref{def_KC} defined in \cite{Kashyap1982},	and \eqref{NewS}) also verify the conditions of Theorem \ref{Theo7}. \\
On the contrary, the AIC criterion with $\textnormal{pen}(m)=\frac 2 n \, |m|$, or the criterion with penalty $\widetilde {\textnormal{pen}}(m)$ do not satisfy these conditions. The following theorem even shows that these criteria asymptotically overfit and have a less good asymptotic efficiency than consistent criteria satisfying Theorem \ref{Theo7}:
\begin{theo}\label{Theo8}
Assume that there exists $g:{\cal M}\to [0,\infty[$ such as $\textnormal{pen}(m)=g(m)/n$ for any $m\in {\cal M}$. Then, under Assumptions {\bf A0}-{\bf A5}, the probability of overfitting is asymptotically positive {\it i.e.} 
\begin{equation}\label{overfitting}
\liminf_{ n \to \infty}  \, \, \P\big (\widehat m_{\textnormal{pen}} ~ \textnormal{overfits}\big )> 0.
\end{equation} 
and there exists $M>0$ such as for $n$ large enough,
\begin{equation}\label{bornewidepen2}
\E \big [  
\ell(\widehat{\theta}_{\widehat{m}_{{\textnormal{pen}}}},\theta^* )\big] \geq \underset{m \in \mathcal{M}}{\min} \, \E \big [   \ell(\widehat{\theta}_{m},\theta^*)  \big ]+ \frac {M}{n} .
\end{equation}
\end{theo}
\begin{corollary}\label{corAIC}
Theorem \ref{Theo8} is valid for $\textnormal{pen}(m)=\frac 2 n \, |m|$ (AIC criterion) or $\textnormal{pen}=\widetilde{\textnormal{pen}}$. 
\end{corollary}
\noindent To conclude, in this context where a true model belonging to a finite family of models exists, the consistent criteria are those which also propose the best asymptotic efficiency. The next section focuses on these. However, a criterion like the AIC criterion or more generally the criterion with penalty $\widetilde{\textnormal{pen}}$ regain all their optimality properties in asymptotic efficiency when the model family is infinite or when the true model does not belong to the family.  
\section{From a Bayesian model selection to a data-driven consistent model selection}\label{sec:bay}
Another classical paradigm for model selection is the Bayesian one, leading typically to the BIC criterion (see \cite{schwarz}). In this approach, the construction of the model selection criterion is first done by assuming that the parameter vector $\theta^*$ is a random vector. Let recall the hierarchical prior sampling scheme in the Bayesian setting: given the finite family of models $\mathcal{M}$, a model $m$ is drawn according to a prior distribution $(\pi_m)_{m \in \mathcal{M}}$ (generally a uniform distribution) and then, conditionally on $m$, $\theta$ is sampled according to some prior distribution $\mu_{m}(\theta)$.\\
~\\
The goal of this model selection procedure is to choose the most probable model after observing the trajectory $X:=(X_1,\cdots, X_n)$, {\it i.e.}
\begin{equation}\label{eq:bic}
\widehat m_B= \underset{m \in \mathcal{M}}{\argmax} \; \big \{\P\big(m\,|\,X\big )\big \}.
\end{equation}
Using Bayes Formula, we can write $\P\big(m\,|\,X\big )=\frac{\pi_m \, \P\big (X\,|\,m\big)}{\P(X)}.$ Moreover, we have:
$$
\P\big (X\, |\, m\big )=\int_{\Theta_m} \P\big (X\,|\, \theta,m\big ) \, d\mu_{m}(\theta).
$$
In addition, since $\P(X)$ does not depend on $m$, and $\P\big (X\,|\, \theta,m\big )$ is the  likelihood of $X$ given $\theta \in \Theta_m$ and $m\in {\cal M}$, maximizing $\P\big (m\,|\, X\big )$ is equivalent to maximize
\begin{equation*}\label{eq:bic2}
	\widehat S_n(m,X):=\log\big(\P\big (X\,|\,m\big ) \big)= \log \Big( \int_{\Theta_m} \pi_m \, \exp\big( L_n(\theta)\big) \,d\mu_{m}(\theta)  \Big).
\end{equation*}
%
From now on, we will assume that $\pi_m=1/|{\cal M}|$ for any $m\in {\cal M}$, a priori uniform distribution of the models in the family ${\cal M}$. We can also assume that there exists a non-negative Borel function $\theta \to b_m(\theta)$ such as $d\mu_{m}(\theta)=b_m(\theta) \, d \theta$. Then we have:
\begin{equation}\label{eq:bic3}
	\widehat S_n(m,X)= -\log (|{\cal M}|)+ \log \Big( \int_{\Theta_m} b_m(\theta)  \, \exp\big( L_n(\theta)\big) \,d\theta  \Big).
\end{equation}
Let us give an asymptotic expansion of the {\it a posteriori} probability in order to derive a BIC type criterion that is coherent with our framework where the observed trajectory is that of a causal affine process. This could be obtained from a Laplace approximation, leading to the following theorem:
\begin{theo}\label{theo:bic}
Under Assumptions {\bf A0, A1, A2, A3, A5} and if for any $x \in R^\infty$, the functions $\theta \to M_\theta$ and $\theta \to f_\theta$ are ${\cal C}^6(\Theta)$ functions satisfying {\bf A}$(\partial^k_{\theta^k}f_{\theta},\Theta)$ and  {\bf A}$(\partial^k_{\theta^k} M_{\theta},\Theta)$ for any $0\leq k \leq 6$. Then
	\begin{multline}\label{eq:bic0}
		\widehat{S}_n(m,X) =\widehat{L}_n(\widehat{\theta}_m)-\frac{\log(n)}{2}\, |m|+\log\big (b_m(\widehat \theta_m)\big ) \\
+\frac{\log(2\pi)} 2 \, |m|-\frac{1}{2}\, \log \big ( \det\big (-\widehat F_n(m)\big)\big) -\log (|{\cal M}|)+O(n^{-1})\quad a.s.
	\end{multline}	
where  $\widehat{F}_n(m) := \big (\partial^2_{\theta_i\theta_j} \widehat{\gamma}_n( \widehat{\theta}_m)\big )_{i,j \in m}$.
\end{theo}
\noindent In the above equation, it is clear that  $-2 \, \widehat S_n(m,X) \simeq -2 \,\widehat{L}_n(\widehat{\theta}_m)+\log(n)\, |m|$ a.s.. 
This gives legitimacy to the usual BIC criterion within the framework of causal affine processes since:
$$
\widehat m_{BIC}=\argmin_{m\in {\cal M}} \Big \{ -2 \,\widehat{L}_n(\widehat{\theta}_m)+\log(n)\, |m|\Big \},
$$
and we see that $\widehat m_{BIC}$ maximizes the main terms of $\widehat S_n(m,X) $.  \\
From the relation \eqref{eq:bic0}, considering certain second order terms of the asymptotic expansion of $\widehat{S}_n(m,X)$, we also obtain the Kashyap criterion (see Kashyap \cite{Kashyap1982}, Sclove \cite{Sclove1987}, Bozdogan \cite{Bozdogan1987}), denoted $KC$ criterion, defined for all $ m \in \cal M$ by 
%
\begin{multline} \label{def_KC}
	\widehat{KC}(m):= -2\,\widehat{L}_n(\widehat{\theta}_m)+\log(n)\, |m|+  \log \big ( \det\big (-\widehat F_n(m)\big)\big)  \\
	\quad\mbox{and}\quad \widehat m_{KC}=\argmin_{m\in {\cal M}} \big \{ \widehat{KC}(m)\big \}.
\end{multline}
Therefore the term  $\displaystyle \log \big ( \det\big (-\widehat F_n(m)\big)\big) $ is added to the usual BIC criterion. Several examples of computations of this term, generally equal to $c\, |m|$ but not always, are provided in the forthcoming Section \ref{Exam}. It is clear that $\widehat m_{KC}$ can be more interesting that $\widehat m_{BIC}$ in terms of consistency only for non asymptotic framework (typically for $n$ of the order of a hundred or several hundred). 
Note also that the data-driven criteria $KC$ that is "optimal" in the sense of the  a posteriori probability (see Kashyap \cite{Kashyap1982}) is also asymptotically consistent under the Assumption {\bf A5}. \\
However, this choice of second order terms of the asymptotic expansion of $\widehat{S}_n(m,X)$ is somewhere arbitrary. A criterion taking account of all the second order terms could also be defined. For this, we could define a uniform distribution $b_m$ on a compact set included in $\Theta_m$. As a consequence, using condition \eqref{Theta} of Assumption {\bf A0}, there always exists $0<C_1\leq C_2$ such as $\frac {C_1}{|m|} \leq b_m(\Theta_m) \leq  \frac {C_2}{|m|}$. As a consequence, we could define a new  data-driven consistent criterion, called $KC'$, such as for any $m\in {\cal M}$
\begin{multline}\label{NewS}
\hspace{-0.4cm}\widehat{KC'}(m):= -2\,\widehat{L}_n(\widehat{\theta}_m)+\big (\log(n)-\log(2\pi) \big )\, |m|+\log \big ( \det\big (-\widehat F_n(m)\big)\big)  +2 \log \big(|m| \big) \\
\qquad\qquad\qquad\mbox{and}\quad \widehat m_{KC'}=\argmin_{m\in {\cal M}} \big \{ \widehat{KC}'(m)\big \}.
\end{multline}
\begin{remark}
We also know that under Assumptions {\bf A0-A5}, $\widehat F_n(m) \limiteasn F(\theta^*_m)$ where $F$ is defined in \eqref{matrixF}. Therefore the term $\log \big ( \det\big (-\widehat F_n(m)\big)\big) $ can also be replaced by $\log \big ( \det\big (-F_m(\theta^*_m)\big)\big)$ in the expression of $\widehat{KC'}(m)$.
\end{remark}
\begin{corollary}\label{cor:bre}
The criteria BIC, KC and KC' are consistent model selection criteria and satisfy Theorem \ref{Theo7}.
\end{corollary}
\noindent Thus, these three criteria are asymptotically consistent and asymptotically efficient following the inequalities \eqref{bornewidepen} and \eqref{borneopti}. Monte-Carlo experiments in Section \ref{sec:simu} will also exhibit that $\widehat m_{KC'}$, which is a data-driven criterion, outperforms $\widehat m_{BIC}$ in terms of consistency and efficiency when the $n$ size of the trajectory is of the order of a hundred or a thousand.
\section{Examples of computations of the asymptotic expectation of ideal penalties}\label{Exam}
From \cite{barW}, with $\mu_4= \E[\xi_0^4]$, $f_\theta^0$ and $H_\theta^0$ defined in \eqref{R}, we have for $m^*\subset m$ and $i,j \in m$:
\begin{eqnarray}
\nonumber  \big (G_m(\theta^*_m)\big )_{i,j}&=& \E\Big[ \frac { \partial_{\theta_i} f_{\theta_m^*}^0\,\partial_{\theta_j} f_{\theta_m^*}^0}{ H_{\theta_m^*}^0} +\frac {(\mu_4 -1)} 4 \,  \frac{ \partial_{\theta_i} H_{\theta_m^*}^0\,\partial_{\theta_j} H_{\theta_m^*}^0}{(H_{\theta_m^*}^0)^{2}}\Big] \\
 \big (F_m(\theta^*_m)\big )_{i,j}&=& - \E\Big[\frac { \partial_{\theta_i} f_{\theta_m^*}^0\,\partial_{\theta_j} f_{\theta_m^*}^0}{ H_{\theta_m^*}^0}  +\frac {1} 2 \,  \frac{ \partial_{\theta_i} H_{\theta_m^*}^0\,\partial_{\theta_j} H_{\theta_m^*}^0}{(H_{\theta_m^*}^0)^{2}}\Big]
\label{eq:m1}
,
\end{eqnarray}
Here there are 3 frameworks where $\mbox{Trace}\Big( \big (F_m(\theta^*_m)\big )^{-1}  \,G_m(\theta^*_m)\Big )$ can be computed for $m^*\subset m$: \\
~\\
\noindent 1/ A first and well-known case is the Gaussian case. Indeed, when $(\xi_t)$ is a Gaussian white noise, then $\mu_4=3$ and then from \eqref{eq:m1}, for any $i,j \in m$,
\begin{multline*}
\big (G_m(\theta_m^*)\big )_{i,j}=- \big (F_m(\theta_m^*)\big )_{i,j} \\ \Longrightarrow ~~-2 \,\mbox{Trace}\big( \big (F_m(\theta^*_m)\big )^{-1}  \,G_m(\theta^*_m)\big )=2 \, \mbox{Trace}\big(I_{|m|}\big )=2 \,|m|, 
\end{multline*}
with $I_\ell$ the identity matrix of size $\ell \in \N^*$. As a consequence, in the Gaussian framework, for $m^*\subset m$, the expectation of the ideal penalty is exactly the classical Akaike Criterion (AIC). \\
~\\
\noindent 2/ A frequent case is when the parameter $\theta$ identifying an affine causal model $ X_t=M_{\theta}^t\, \xi_t+f_{\theta}^t$ can be decomposed as $\theta=(\theta_1, \theta_2)'$ with  $f_{\theta}^t=\widetilde{f}^t_{{\theta}_1}$ and $M_{\theta}^t=\widetilde{M}^t_{{\theta}_2}$. Let $p_1,p_2$ such that $p_1=|\theta_1|$, $p_2=|\theta_2|$ and $|m|=p_1+p_2$.\\
In such a case, from (\ref{eq:m1}), it is clear that all the terms $F_m(\theta_m^*)_{i,j}$ and $G_m(\theta_m^*)_{i,j}$ are equals to zero for $i=1,\ldots, p_1$ and $j=1,\cdots, p_2$ implying
$$
F_m(\theta_m^*)= -\left (
\begin{array}{cc}
A_{1,p_1} &  O_{p_1,p_2} \\
O_{p_2,p_1}  &  B_{p_1+1,p_1+p_2}
\end{array}
\right )  ~
\mbox{and} \quad
G_m(\theta_m^*)= \left (
\begin{array}{cc}
A_{1,p_1} &  O_{p_1,p_2} \\
O_{p_2,p_1}  &  \frac {(\mu_4-1)}2\, B_{p_1+1,p_1+p_2}
\end{array}
\right )
$$
where $O$ is the null matrix and from the expressions of matrix $G_m(\theta^*_m)$ and $F_m(\theta^*_m)$ in \eqref{eq:m1},   \\
$A_{1,p_1}=\Big(\E\Big[ \frac { \partial_{\theta_i} f_{\theta_m^*}^0\,\partial_{\theta_j} f_{\theta_m^*}^0}{ H_{\theta_m^*}^0} \Big]\Big)_{1\le i,j\le p_1} $\hspace{-0.4cm} and $\quad	 B_{p_1+1,p_1+p_2}=\Big(\frac 1 2 \, \E\Big[ \frac{ \partial_{\theta_i} H_{\theta_m^*}^0\,\partial_{\theta_j} H_{\theta_m^*}^0}{(H_{\theta_m^*}^0)^{2}}\Big]\Big)_{p_1+1\le i,j\le p_1+p_2}$. As a consequence, 
\begin{align*}
G_m(\theta_m^*)\,F_m(\theta_m^*)^{-1}& =-\mbox{Diag}\Big (A_{1,p_1},\frac{(\mu_4-1)}2\, B_{p_1+1,p_1+p_2}\Big ) \times \mbox{Diag}\big (A_{1,p_1}^{-1}, B_{p_1+1,p_1+p_2}^{-1}\big)\\
&=- \mbox{Diag}\big (I_{p_1},\frac{(\mu_4-1)}2\, I_{p_2} \big )
\end{align*}
and we obtain
\begin{equation}\label{Trace1}
-2 \, \mbox{Trace}\Big( \big (F_m(\theta^*_m)\big )^{-1}  \,G_m(\theta^*_m)\Big )= 2\, p_1+ (\mu_4-1)\,p_2.
\end{equation}
This setting includes many classical times series: 
\begin{itemize}
	\item For ARMA$(p,q)$ processes, we have $X_t=f_{\theta}^t+\sigma\, \xi_t$ since $X_{t}+a_1\, X_{t-1}+ \cdots +a_p \, X_{t-p}=\sigma\, \big (\xi_t+b_1\, \xi_{t-1}+\cdots+b_q \, \xi_{t-q}\big )$ for all $t \in \Z$. Then $\theta_1=\big (a_1,\ldots,a_p,b_1,\ldots,b_q) $ and $\theta_2=\sigma$. The penalty term is slightly different according to $\sigma$ is known or not:
	\begin{enumerate}
		\item[(a)] if $\sigma$ is known, then $\theta=\theta_1$ and $G_m(\theta^*)=- F_m(\theta^*)$, so that we recover exactly the AIC  penalty term:
		$$
		-2 \, \mbox{Trace}\big(G_m(\theta^*_m)F_m(\theta^*_m)^{-1} \big)=2\, |m|=2\,(p+q);
		$$
		\item[(b)] if $\sigma$ is unknown,  $\theta=(\theta_1,\sigma)$ and simple computations lead to 
		\begin{multline*}	
		F_m(\theta^*)= \left (
		\begin{array}{cc}
		(F_m(\theta^*)\big )_{1\le i,j \le |m|-1} &  0 \\
		0  &  -\frac{1}{2\,\sigma^4}
		\end{array}
		\right )  \\
\mbox{and}~
		G_m(\theta^*)= \left (
		\begin{array}{cc}
		(G_m(\theta^*)\big )_{1\le i,j \le |m|-1} &  0 \\
		0  &  \frac{(\mu_4-1)}{4\, \sigma^4}
		\end{array}
		\right )
		\end{multline*}
		where $(G_m(\theta^*)\big )_{1\le i,j \le |m|-1}= - (F_m(\theta^*)\big )_{1\le i,j \le |m|-1}$.\\ Thus, 
		we obtain
		$\displaystyle G_m(\theta^*)F_m(\theta^*)^{-1}= -\left (
		\begin{array}{cc}
		I_{1\le i,j \le |m|-1} &  0 \\
		0  &  \frac {\mu_4-1}2
		\end{array}
		\right ) $
		and therefore, with $|m|=p+q+1$ in this case,
$$
	 -2 \, \mbox{Trace}\big(G_m(\theta^*_m)F_m(\theta^*_m)^{-1} \big)= 2\, |m|+ (\mu_4-3)=2(p+q)+(\mu_4-1),
	$$
	and therefore once again the expectation of the ideal penalty leads to the AIC model selection.
	\end{enumerate}
	\item For GARCH$(p,q)$ processes (see \cite{francq2019}), we have $f_\theta=0$ and $X_t=M_\theta^t \, \xi_t$ since for any $t\in \Z$,
	$$
	\left\{ \begin{array}{ccl} X_t&=&\sigma_t \, \xi_t\\
	\sigma_t^2&=&\omega_0+a_1\, X_{t-1}^2+\cdots+ a_p\, X_{t-p}^2+b_1\, \sigma_{t-1}^2+\cdots+b_q\, \sigma_{t-q}^2
	\end{array} \right .  .
$$
Denote $\theta=\theta_2=(\omega_0,a_1,\ldots,a_p,b_1,\ldots,b_q)$. \\
Then we have $A_{p_1}=0$
 and therefore $G_m(\theta^*)=- \frac {(\mu_4-1)}2 \, F_m(\theta^*).$ As a result:
	\begin{equation*}
-2 \, \mbox{Trace}\big(G_m(\theta^*_m)F_m(\theta^*_m)^{-1} \big)=(\mu_4-1)\, |m|=(\mu_4-1)\, (p+q+1).
	\end{equation*}
		\item For APARCH$(\delta,p,q)$ processes (see \cite{ding}), we also have $f_\theta=0$ and $X_t=M_\theta^t \, \xi_t$ since for any $t\in \Z$,
		$$\left\{ \begin{array}{ccl} X_t&=&\sigma_t \, \xi_t\\
	\sigma_t^\delta&=&\omega_0+a_1\, (X_{t-1}-\gamma_1|X_{t-1}|)^\delta+\cdots+ a_p\, (X_{t-p}-\gamma_p|X_{t-p}|)^\delta \\
	&& \qquad \hspace{5.5cm}+b_1\, \sigma_{t-1}^\delta+\cdots+b_q\, \sigma_{t-q}^\delta
	\end{array} \right . .
$$
For such a process, $\theta=\theta_2=(\omega_0,a_1,\ldots,a_p,\gamma_1,\ldots,\gamma_p,b_1,\ldots,b_q)$ when we assume that $\delta$ is known, and, {\it mutatis mutandis}, the result is the same than for GARCH processes:
\begin{equation*}
-2 \, \mbox{Trace}\big(G_m(\theta^*_m)F_m(\theta^*_m)^{-1} \big)=(\mu_4-1)\, |m|=(\mu_4-1)\, (2p+q+1).
	\end{equation*}
\end{itemize}
\noindent 3/ Otherwise, the computations are no longer easy. Let us see the example of the family of $AR(1)-ARCH(p)$ processes. Then
for any $t \in \Z$ we have $X_t=\phi X_{t-1}+Z_t $ where $Z_t=\xi_t\, \big (\alpha_0+ \alpha_1 Z_{t-1}^2+\cdots+\alpha_p Z_{t-p}^2\big )^{1/2}$. As a consequence, with $\theta=(\phi, \alpha_0,\ldots, \alpha_p)'$, we obtain for any $t\in \Z$,
\begin{multline*}
X_t=f_{\theta}(X_{t-1})+M_{\theta}(X_{t-1},\dots,X_{t-p-1})\,\xi_t \\
\mbox{with}\qquad \left \{ \begin{array}{lcl}
f_{\theta}(X_{t-1})&= &\phi \, X_{t-1} \\
 M_{\theta}(X_{t-1},\dots,X_{t-p})&=&\big (\alpha_0+ \sum_{i=1}^p \alpha_i (X_{t-i}-\phi X_{t-i-1})^2\big )^{1/2}
\end{array} \right ..
\end{multline*}
Thus the parameter $\phi$ is present in $f_\theta$ as well as in $M_\theta$.
From (\ref{eq:m1}), and with the notations of 1/, we obtain:
\begin{multline*} 
F_m(\theta_m^*)= -\left (
\begin{array}{cc}
A_{1,1}&  O_{1,p+1} \\
O_{p+1,1}  &  O_{p+1,p+1}
\end{array}
\right )-B_{1,p+2} \\
\mbox{and} \quad
G_m(\theta_m^*)=  \left (
\begin{array}{cc}
A_{1,1} &  O_{1,p+1} \\
O_{p+1,1}  &  O_{p+1,p+1}
\end{array}
\right )+ \frac {(\mu_4-1)} 2\,B_{1,p+2}.
\end{multline*}
As a consequence, 
$$
G_m(\theta_m^*)=-\frac {(\mu_4-1)} 2 \, F_m(\theta_m^*)+\frac {(\mu_4-3)} 2\, \left (\begin{array}{cc} A_{1,1}&  O_{1,p+1} \\ O_{p+1,1}  & O_{p+1,p+1} \end{array} \right ). 
$$  
Thus, with $|m|=p+2$, 
$$
G_m(\theta_m^*)\, F_m^{-1}(\theta_m^*)=-\frac {(\mu_4-1)}2\, I_{|m|}+\frac {(\mu_4-3)} 2\,\left (
\begin{array}{cc}
A_{1,1}&  O_{1,p+1} \\
O_{p+1,1}  & O_{p+1,p+1}
\end{array}
\right )\,  F_m^{-1}(\theta_m^*) .  $$ 
Whatever the matrix $F_m^{-1}(\theta_m^*)$, we have $\left (
\begin{array}{cc}
A_{1,1}&  O_{1,p+1} \\
O_{p+1,1}  & O_{p+1,p+1}
\end{array}
\right )\,  F_m^{-1}(\theta_m^*) =\left (
\begin{array}{cc}
c(\theta^*_m)&  O_{1,p+1} \\
O_{p+1,1}  & O_{p+1,p+1}
\end{array}
\right )  $ with $c(\theta^*_m) =c(\theta^*)\in \R$ since $m^*\subset m$. Then for all $m^*\subset m$,
$$-2 \, \mbox{Trace}\big(G_m(\theta^*_m)F_m(\theta^*_m)^{-1} \big)=-2 \,c(\theta^*) +(\mu_4-1)\,|m|,$$
where $-2 \,c(\theta^*)$ does not depend on $m$.

\section{Numerical Studies}\label{sec:simu}
This section aims to investigate the numerical behavior of the model selection criteria studied in Section \ref{sec:resul} and Section \ref{sec:bay} using {\tt R} software. 

\noindent To do that, three Data Generating Processes (DGP) have been considered:\\
\begin{center}
\begin{tabular}{lll} 
	\mbox{DGP I} &\mbox{AR}(2)   & $X_t=0.4\,X_{t-1}+0.4\,X_{t-2}+\xi_t$,\\
	\mbox{DGP II} & \mbox{ARMA}(1,1)   & $X_t-0.5\, X_{t-1}=\xi_t+0.6\, \xi_{t-1}$,\\
	\mbox{DGP III}& \mbox{GARCH}(1,1)  & $X_t=\sigma_{t} \, \xi_t$ \quad \mbox{with} \; $\sigma_{t}^2=1+0.35X_{t-1}^2+0.4\sigma_{t-1}^2$,
\end{tabular}
\end{center}
where $(\xi_t)_t$ is a Gaussian white noise with variance unity. 
\begin{remark} \label{Ident2}
As already observed in the Remark \ref{Ident}, Assumption {\bf A1} is never satisfied for ARMA processes in case of overfitting. However, in the used optimization under constraint algorithm (program {\tt nloptr}), we initialized $\theta$ at $0$ (except for the variance estimator). By this way, we have noticed in Monte-Carlo experiments that the algorithm always converges to $\theta^*$ and not other solution due to the overfitting.
\end{remark}
\noindent In order to illustrate the obtained theoretical asymptotic behaviors, we have realized Monte-Carlo experiments where the performance of the AIC, BIC and KC' criteria are compared using the following parameters:
\begin{itemize}
\item The considered family of competitive models is the same for the three DGP
$$\mathcal{M}=\big \{\mbox{ARMA}(p,q) \;\;\mbox{and}\;\; \mbox{GARCH}(p,q) \;\mbox{processes}\;\; \mbox{with}\;\;0 \le p,q \le 6\big\}.$$ 
\item Several values of $n$, the observed trajectory length, are considered: {200, 500, 1000, 2000}.
\item For each $n$ and DGP, we have generated 500 independent replications of the trajectories.
\end{itemize}
Hence, for each replication, the selected models $\widehat m_{AIC}$, $\widehat m_{BIC}$ and $\widehat m_{KC'}$ are computed. 
Then, 
\begin{enumerate}
\item The consistency property is illustrated by the computation of the frequency (percentage) of selecting  the true model versus a model other than the true one (called here "wrong"). 
\item For the efficiency property (Theorem \ref{theo:1}, \ref{Theo7} and \ref{Theo8}), we first compute a very sharp estimator $\widetilde R$ of the risk function $R$ for each DGP: $\widetilde R=\widehat{\gamma}_N$ computed from an independent and very large ($N=10^6$) trajectory of the DGP. By this way, and we obtain an estimator $\widetilde {\ell}(\widehat{\theta}_{\widehat m},\theta^*)=\widetilde R(\widehat{\theta}_{\widehat m})-\widetilde R(\theta^*)$ of $ {\ell}(\widehat{\theta}_{\widehat m},\theta^*)$ for $\widehat m=\widehat m_{AIC}$, $\widehat m_{BIC}$ and $\widehat m_{KC'}$. Then, we compute
\[
\widehat ME:= n\,\Big(\overline{\widetilde {\ell}(\widehat{\theta}_{\widehat m},\theta^*)}-\overline{\widetilde {\ell}(\widehat{\theta}_{m^*},\theta^*)}\Big)
\]
where $\overline{\widetilde \ell~}$ is the average of $\widetilde \ell$ over the 500 replications. Therefore $\widehat{ME}$ is an estimator of $n\, \big ( \E\big [{\ell}(\widehat{\theta}_{\widehat m},\theta^*) \big ]-\underset{m \in \mathcal{M}}{\min} \E\big [{\ell}(\widehat{\theta}_{m},\theta^*) \big ] \big)$, which appears in 
\eqref{borneopti} and \eqref{bornewidepen2}.
\end{enumerate} 
\noindent The results of Monte-Carlo experiments are reported in Table \ref{tab:2}, devoted to the consistency property, and in Table \ref{tab:21}, devoted to the efficiency property. \\
\begin{table}
\begin{center}
	\begin{tabular}{|l|c|lcr|lcr|lcr|lcr|}
		\toprule
	\multicolumn{1}{|c|}{}&	\multicolumn{1}{|c|}{n}& \multicolumn{3}{c}{$200$ } & \multicolumn{3}{c}{$500$ }& \multicolumn{3}{c}{$1000$ } & \multicolumn{3}{c|}{$2000$ }\\
	&	&  AIC& BIC & KC'  &  AIC& BIC & KC' &   AIC& BIC & KC' &  AIC& BIC & KC'\\
		\midrule
	DGP I &	True &17.2    &36.2   &35.6 &30.4 &73.2 &78.2    &36.4    &87.4  & 92.2  & 32.4 &96.2 &98.4 \\
	&	Wrong &82.8     &63.8    &64.4  &69.6   &26.8 &21.8  &63.6  &13.6 & 7.8& 67.6 &03.8 &01.6\\
		\hline
	 \midrule
		DGP II &	{ True} & 27.8 & 80.8 &92.0 & 30.6 &88.4 &96.6 & 31.0 & 89.1 & 97.5 & 33.3 &95.2 &99.9\\
		&{Wrong} & 72.2 &19.2 &08.0 &69.7 & 11.6& 03.4 & 69.0 & 10.9 & 02.5 & 66.7 &04.8 &00.1\\
		\hline
		\midrule
		DGP III &	True &00.4 &10.8 &14.8 &01.4 &32.2 &55.8 &01.0  &54.8 & 82.0 &02.0 &75.8&93.8 \\
	&	Wrong &99.6 &89.2 &85.2 &98.6 &67.8 &44.2  &99.0  &45.2 &18.0&98.0 &24.2 &06.2 \\
		\bottomrule
	\end{tabular}
\end{center}
\caption{Percentage of "true" selected models depending on the criterion and sample's length for DGP I-III.} \label{tab:2}
\end{table}
\vspace{0.6cm}

\begin{table}
\begin{center}
	
	\begin{tabular}{|c|lcr|lcr|lcr|lcr|}
		\toprule
		\multicolumn{1}{|c|}{n}& \multicolumn{3}{c}{$200$ } & \multicolumn{3}{c}{$500$ }& \multicolumn{3}{c}{$1000$ } & \multicolumn{3}{c|}{$2000$ }\\
		&  AIC& BIC & KC'  &  AIC& BIC & KC' &   AIC& BIC & KC' &  AIC& BIC & KC'\\
		\midrule
	
		DGP I &4.91     &2.59    &5.35  &3.46   &1.11 &1.18  &3.08  &0.98 & 0.75& 3.05 &0.38 &0.29\\
		\hline
		\midrule
	
		DGP II & 3.66 &0.87 &0.54 &3.37 & 0.42& 0.11 & 2.62 & 0.15 & 0.05 & 2.5 &0.10 &0.04\\
		\hline
		\midrule
		
		DGP III &2.39 &4.63 &13.16 &2.53 &4.08 &9.54  &2.69  &2.96 &2.52&3.21 &2.06 &0.76 \\
		\bottomrule
	\end{tabular}
\end{center}
\caption{$\widehat {ME}$ of selected models depending on the criterion and the sample's length for DGP I-III.} \label{tab:21}
\end{table}
\noindent {\bf Conclusions of numerical experiments:} 
\begin{itemize}
\item Concerning the consistency properties, the numerical results of Table \ref{tab:2} show that the percentages of choice of the true model tend towards 100 for increasing $n$ and with the criteria BIC and KC', and this corresponds well to the obtained asymptotic result (Corollary \ref{cor:bre}). And as it could also be deduced from the theory (see Corollary \ref{corAIC}) the AIC criterion is not a consistent one. Moreover, we observe that the KC' criterion outperforms BIC when dealing as well as small and large samples for all considered DGP. These results confirm that it is important to also consider the neglected terms in the derivation of the BIC criterion. 

\item From the results of Table \ref{tab:21}, we notice a decrease of the residual term $\widehat{ME}$  to $0$ for increasing $n$ for the consistent criteria BIC and KC'. This corresponds well to the $o(1/n)$ term observed in \eqref{borneopti}. We also observe that this convergence to 0 is globally faster with the KC' criterion than with the BIC one. Thus, in terms of efficiency as well as in terms of consistency, the KC' criterion performs even better than the BIC one for the selected DGPs. Finally, as shown by Theorems \ref{theo:1} and \ref{Theo8}, the  statistic $\widehat{ME}$ seems asymptotically bounded and does not converge to $0$ when the AIC criterion is applied to select the model.  This confirms that BIC and especially KC' criteria are more accurate in terms of efficiency than AIC criterion.
\end{itemize}

\section{Proofs}\label{sec:proof}
\subsection{Proofs of Section \ref{TLCm}}
The asymptotic normality of $\big ( \frac 1 n \big (\partial_{\theta_i} L_n(\theta_m^*) \big )_{i\in m}$ was established in \cite{barW}  and \cite{bar2019} when $m^*\subset m$ using a central limit theorem for stationary martingale difference. Here we extend this result to any $m\in {\cal M}$:
\begin{proposition} \label{theo0}
Under Assumption {\bf A0-A5}, for any $\theta \in \Theta$, we have
\begin{multline}
\sqrt{n} \, \Big (\frac 1 n \,  \partial _{\theta} L_n(\theta) + \frac 1 2 \, \E \big [\partial _{\theta} \gamma(\theta,X_0) \big ]   \Big ) \limiteloin {\cal N} \big ( 0 \, , \, G(\theta) \big ) \\
 \mbox{with}\quad G(\theta):=\frac 1 4 \, \Big (\sum_{t\in \Z} \cov \big ( \partial_{\theta_i} \gamma(\theta,X_0)\, , \, \partial_{\theta_j} \gamma(\theta,X_t)\big )\Big )_{1\leq i,j\leq d}.
\label{clt0}
\end{multline}
\end{proposition}
\noindent The main tool we use here for establishing Theorem \ref{theo0} is the notion of $\tau$-dependence for stationary time series and more precisely, the $\tau$-dependence coefficients, which are a version of the coupling coefficients introduced in \cite{dp} and used for stationary infinite memory chains. The reader is deferred to the  lecture notes \cite{DDLLLP}  for complements and details on coupling, based on the Wasserstein distance between probabilities defined as below. Its stationary version is:
\begin{definition}\label{defwd}
Let $(\Omega,\mathcal{C}, \P)$ be a probability space,
$\mathcal{M}$ a $\sigma$-subalgebra of $\mathcal{C}$ and $Z$ a
random variable with values in $E$. Assume that $\|Z\|_p<\infty$ and define the coefficient $\tau^{(p)}$ as
\begin{equation*}
\tau^{(p)}(\mathcal{M},Z)=
\Big \|\sup_{f\in\Lambda_1(E)}\Big\{\Big|\int f(x)\P_{Z|\mathcal{M}}(dx)-\int
f(x)\P_{Z}(dx) \Big |\Big \}\Big \|_p.
\end{equation*}
\end{definition}
\noindent Using the definition of $\tau$, the dependence between the past of the sequence $(Z_t)_{t\in\Z}$ and its future $k$-tuples may be assessed: consider the norm
$\|x-y\|=\|x_1-y_1\|+\cdots+\|x_k-y_k\|$ on $E^k$,
set $\mathcal{M}_p=\sigma(Z_t,t\le p)$ and define
\begin{eqnarray*}
\tau_Z^{(p)}(s)&=&\sup_{k>0} \Big \{ \max_{1\le l\le k}
\frac1l\sup\Big\{\tau^{(p)}(\mathcal{M}_p,(Z_{j_1},\ldots,Z_{j_l}))\mbox{ with }p+s\le
j_1<\cdots <j_l\Big\} \Big \}.
\end{eqnarray*}
Finally, the time series $(Z_t)_{t\in\Z}$ is {\it $\tau_Z^{(p)}$-weakly dependent} when its coefficients $\tau_Z^{(p)}(s)$ tend to $0$ as $s$ tends to infinity.
\begin{lemma}\label{lem0}
Under Assumption {\bf A0}, then for $p\leq r$ and $b^{(p)}_k= \alpha_k(M_\theta,\Theta) \, \| \xi_0\|_p +\alpha_k(f_\theta,\Theta)$ for any $j\in \N^*$, 
\begin{equation}\label{tauX} 
\tau_{X}^{(p)} (s) \le C \,\lambda_s \quad\mbox{with}\quad \lambda_s=  \inf_{1\leq r\le s} \Big \{ \Big (\sum_{k=1}^\infty b^{(p)}_k \Big )^{s/r}+\sum_{t= r+1}^\infty b^{(p)}_t \Big\}\quad \mbox{for $s \ge 1$}.
\end{equation}
\end{lemma}
\begin{proof}[Proof of Lemma \ref{lem0}]
This Lemma can be directly deduced from Proposition 3.1  of \cite{dou} where $T(x,\xi_0)=M_\theta(x)\, \xi_0 +f_\theta(x)$ for any $x \in \R^\infty$ and therefore 
$$\big \|T(x,\xi_0)-T(y,\xi_0) \big \|_p\leq  \|\xi_0\| _p \, \big | M_\theta(x)-M_\theta(y) \big |+\big | f_\theta(x)-f_\theta(y) \big |
$$ inducing $\big \|T(x,\xi_0)-T(y,\xi_0) \big \|_p \leq \sum_{k=1}^\infty b^{(p)}_k$.
\end{proof}
\begin{remark}
Using Assumption {\bf A0} and {\bf A5}, we deduce that $b^{(p)}_t =O\big (t^{-\delta}\big )$ with $\delta>7/2$, and therefore  $\tau_{X}^{(p)} (s) \leq \lambda_s =O\big (s^{1-\delta} \log s\big )$.  
\end{remark}
\noindent Now, under the Assumption {\bf A0}, since $X$ is a causal time series, define for any $j=1,\ldots,d$ and $\theta \in \Theta$,
$$
\phi^{(j)}_\theta\big ( (X_{t-k})_{k\geq 0} \big ):=\partial_{\theta_j} \gamma(\theta,X_t)=-2 \, \partial_{\theta_j} M^t_\theta \, \frac {(X_t-f_\theta^t)^2}{\big (M^t_\theta\big )^3}- 2 \, \partial_{\theta_j} f^t_\theta \, \frac {X_t-f_\theta^t}{\big (M^t_\theta\big )^2}+ 2 \, \frac {\partial_{\theta_j} M^t_\theta } {M^t_\theta}.
$$
Then we have:
\begin{lemma} \label{lem00}
Under Assumption {\bf A0-A5}, for any $j=1,\ldots, d$, for any $\theta \in \Theta$, the sequence $\big (\phi^{(j)}_\theta\big ( (X_{t-k})_{k\geq 0} \big )\big )_{t\in \Z}$ is a causal stationary sequence that is  $\tau_{\phi^{(j)}_\theta}^{(p)}$-weakly dependent where its coefficients $\tau_{\phi^{(j)}_\theta}^{(1)}(s)$ satisfies:
\begin{multline}\label{tauphi} 
\tau_{\phi^{(j)}_\theta}^{(1)}(s) \le C \,\Big ( \sum_{\ell=1}^{s} \big (\alpha_\ell(f_\theta,\Theta)+\alpha_\ell(M_\theta,\Theta)+\alpha_\ell(\partial_{\theta_j} M_\theta, \Theta) +\alpha_\ell(\partial_{\theta_j} f_\theta, \Theta) \big )\,\lambda_{s+1-\ell} \\+ \sum_{\ell=s+1}^{\infty} \big (\alpha_\ell(f_\theta,\Theta)+\alpha_\ell(M_\theta,\Theta)+\alpha_\ell(\partial_{\theta_j} M_\theta, \Theta) +\alpha_\ell(\partial_{\theta_j} f_\theta, \Theta) \big ) \Big ),
\end{multline}
$\mbox{for any  $s \ge 0$}$ where $(\lambda_s)$ is defined in \eqref{tauX}.
\end{lemma}
\begin{proof}[Proof of Lemma \ref{lem00}]
In the proof of Proposition 4.1 of \cite{bdw}, it has been proven for $U=(U_i)_{i\geq 1}$ and $V=(V_i)_{i\geq 1}$ such as $\sup_{i \geq 1} \big \{ \big \|U_i\|_4 \vee \big \|V_i\|_4 \big \}<\infty$ that there exists $C>0$ satisfying
\begin{multline} \label{EE}
\E \big [ \sup_{\theta\in \Theta} \big \|\phi_\theta^{(j)}(U)  -\phi_\theta^{(j)}(V)\big \|\; \big ] 
\leq C 
\, \Big (\|U_1-V_1\|_4 \\
+\sum_{i=2}^\infty \big ( \alpha_i(f_\theta,\Theta)+\alpha_i(M_\theta,\Theta)+\alpha_i(\partial_{\theta_j}f_\theta ,\Theta)+\alpha_i(\partial_{\theta_j}M_\theta ,\Theta)\big ) \|U_i-V_i\|_4 \Big ).
\end{multline}
Using coupling techniques, if  $(\widetilde \xi_t)_{t\in \Z}$ is an independent replication of $(\xi_t)_{t\in \Z}$, define also $(\widetilde X_t)_{t\in\Z}$ satisfying the assumptions with $(\widetilde \xi_t)_{t\in \Z}$ instead of $( \xi_t)_{t\in \Z}$ and $\big ( \phi^{(j)}_\theta \big ( (\widetilde X_{t-k})_{k\geq 0} \big )\big )_{t\in \Z}$.
Then for $s\geq 0$, using \eqref{EE},
\begin{eqnarray*}\label{taucouplephi}
\tau_{\phi^{(j)}_\theta}^{(1)} (s) &\leq &\big \| \phi^{(j)}_\theta \big ( ( X_{s-k})_{k\geq 0} \big )- \phi^{(j)}_\theta \big ( (\widetilde X_{s-k})_{k\geq 0} \big )\big \|_1 \\
&\leq & C 
\, \Big (\|X_1-\widetilde X_1\|_4 \\
&& \qquad +\sum_{i=2}^\infty \big ( \alpha_i(f_\theta,\Theta)+\alpha_i(M_\theta,\Theta)+\alpha_i(\partial_{\theta_j}f_\theta ,\Theta)+\alpha_i(\partial_{\theta_j}M_\theta ,\Theta)\big ) \|X_i-\widetilde X_i\|_4 \Big ) \\
& \leq & C \, \sum_{\ell=1}^\infty \big ( \alpha_\ell(f_\theta,\Theta)+\alpha_\ell(M_\theta,\Theta)+\alpha_\ell(\partial_{\theta_j}f_\theta ,\Theta)+\alpha_\ell(\partial_{\theta_j}M_\theta ,\Theta)\big )\, \lambda_{s+1-\ell},
\end{eqnarray*}
that implies \eqref{tauphi}.
\end{proof}
\begin{remark}
Under Assumption {\bf A0} and {\bf A5}, and therefore with $\lambda_s=O\big (s^{1-\delta} \log s\big )$ with $\delta>7/2$, we also deduce that $\tau_{\phi^{(j)}_\theta}^{(1)} (s)  =O\big (s^{1-\delta} \log s\big )$. 
\end{remark}
\begin{proof}[Proof of Proposition \ref{theo0}]
If $Z$ is a $\tau_{Z}$-dependent centered stationary time series satisfying $\E[|Z_0|^\kappa ]<\infty$ with $\kappa>2$, and $\sum_{s=1}^\infty s^{1/(\kappa-2)} \tau_{Z}(s)<\infty$, we deduce from Lemma 2, point 2. of \cite{dd} that condition D$(2,\theta/2,X)$  is satisfied  as $\theta$-weakly dependent coefficients are smaller than $\tau$-weakly dependent coefficients, see (2.2.13) p.16 of \cite{DDLLLP}, and $0<\sum_{t\in \Z}\big | \E [Z_0Z_t] \big |<\infty$ from Proposition 2 of \cite{dd}. Then,
\begin{equation*}
 \frac 1 {\sqrt{n} } \, \sum_{t=1}^{n} Z_t   \limiteloin {\cal N} \Big ( 0 \, , \,  \sum_{t\in \Z} \E [Z_0Z_t] \Big ).
\end{equation*} 
We can apply this central limit theorem to 
$$
Z_t:=\sum_{j=1}^d c_j\big (\phi_\theta^{(j)}\big ((X_{t-k})_{k\geq 0} \big )-\E\big [\phi_\theta^{(j)}\big ((X_{t-k})_{k\geq 0} \big ) \big ]\big )\quad\mbox{with $(c_j)_{1\leq j\leq d} \in \R^d$}.
$$
Indeed, using Lemma \ref{lem00}, we easily obtain for $s\geq 0$
$$
\tau_{Z}(s)\leq  C \,\Big ( \sum_{j=1}^d |c_j| \Big ) \, \sum_{\ell=1}^\infty \big ( \alpha_\ell(f_\theta,\Theta)+\alpha_\ell(M_\theta,\Theta)+\alpha_\ell(\partial_{\theta_j}f_\theta ,\Theta)+\alpha_\ell(\partial_{\theta_j}M_\theta ,\Theta)\big )\, \lambda_{s+1-\ell},
$$
and therefore under Assumption {\bf A0} and {\bf A5}, $\tau_{Z}(s)= O\big (s^{1-\delta} \log s\big )$. 
Moreover, using Lemma \ref{lem:7}, we deduce $\E \big [|Z_0|^{8/3}\big ]<\infty$. Then with $\kappa=8/3$,  $\sum_{s=1}^\infty s^{1/(\kappa-2)} \tau_{Z}(s)=\sum_{s=1}^\infty s^{3/2} \,  \tau_{Z}(s)<\infty$ is satisfied since $\delta>7/2$. \\
Therefore, we deduce for any $\theta \in \Theta$,
\begin{multline*}
\sqrt{n} \,\sum_{j=1}^d c_j \, \Big (\frac 1 n \,  \partial _{\theta_j} L_n(\theta) + \frac 1 2 \, \E \big [\partial _{\theta_j} \gamma(\theta,X_0) \big ]   \Big ) \\ \limiteloin {\cal N} \Big ( 0 \, , \, \frac 1 4 \, \sum_{i=1}^d \sum_{j=1}^dc_i \,   c_j \sum_{t\in \Z} \cov \big ( \partial_{\theta_i} \gamma(\theta,X_0)\, , \, \partial_{\theta_j} \gamma(\theta,X_t)\big )\Big ),
\end{multline*} 
which implies the multidimensional central limit theorem \eqref{clt0}. 
\end{proof}
\begin{proof}[Proof of Corollary \ref{corol}]
Firstly, it was already established in \cite{bar2019}  that if $m^*\subset m$  then $\big (\partial_{\theta_i} \gamma(\theta_m,X_t) \big )_{t\in \Z}$ is a stationary martingale difference process with respect to ${\cal F}_t=\sigma\big ( (X_{t-k})_{k\in \N} \big )$. As a consequence $\cov \big ( \partial_{\theta_i} \gamma(\theta,X_0)\, , \, \partial_{\theta_j} \gamma(\theta,X_t)\big )=0$ if $t\neq 0$. \\
Secondly, for all $m \in {\cal M}$, from the definition of $\theta^*_m$ as a local minimum of $R$ on $\Theta_m$, and from Assumption {\bf A0-A5}, then $\partial_{\theta_j} R(\theta^*_m) =\E \big [\partial_{\theta_j} \gamma(\theta^*_m,X_0) \big ]=0$ for all $j\in m$.
\end{proof}
\begin{proof}[Proof of Theorem \ref{theomiss}]
We use here a standard proof, allowing to show the asymptotic normality of the QMLE and already used in \cite{barW}. \\
Firstly, it was established in \cite{bar2019} that $\widehat \theta_m \limiteasn \theta_m^*$. \\
Secondly, a Taylor-Lagrange expansion is applied to $\big ( \partial_{\theta_j} L_n(\widehat \theta_m) \big )_{j\in m}$ around $\theta^*_m$:
\begin{multline}\label{antheta}
\frac 1 {\sqrt n} \,  \big ( \partial_{\theta_j} L_n(\widehat \theta_m) \big )_{j\in m}=\frac 1 {\sqrt n} \,  \big ( \partial_{\theta_j} L_n(\theta^*_m) \big )_{j\in m} \\
+ \big ( \frac 1 n \, \partial^2_{\theta_i \theta_j} L_n(\overline \theta_m) \big )_{i,j\in m} \times \sqrt n \, \big ( (\widehat \theta_m)_i-(\theta^*_m)_i \big )_{i \in m} 
\end{multline}
with $\overline \theta_m=c\, \widehat \theta_m+ (1-c)\theta^*_m$ and $0<c<1$.\\
Using $\widehat \theta_m \limiteasn \theta_m^*$ and the ergodic theorem $\frac 1 n \, \big (\partial^2_{\theta_i \theta_j} L_n(\theta_m)\big)_{i,j \in m}\limiteasn F_m(\theta_m)$ for any $\theta_m \in \Theta_m$ since $\E \big [\big \| \partial^2 _{\theta^2} \gamma(\theta,X_0) \big \|_\Theta \big ]<\infty$ , we obtain:
\begin{equation}\label{L2F}
\Big ( \frac 1 n \, \partial^2_{\theta_i \theta_j} L_n(\overline \theta_m) \Big )_{i,j\in m} \limiteasn F_m(\theta^*_m).
\end{equation}
Finally, by definition of $\widehat \theta_m$, $\partial_{\theta_j}\widehat L_n(\widehat \theta_m)=0$ for any $j\in m$. As a consequence, 
\begin{equation}\label{diffL}
\frac 1 {\sqrt n} \,  \big ( \partial_{\theta_j} L_n(\widehat \theta_m) \big )_{j\in m} \limiteproban 0,
\end{equation}
using a Markov Inequality and $\E \big [ \frac 1 {\sqrt n} \, \big \| \partial_{\theta}\widehat L_n(\theta)-\partial_{\theta}L_n(\theta) \|_{\Theta} \big ]\limiten 0$ established in (5.11) of \cite{barW}. Considering \eqref{antheta}, \eqref{L2F} and \eqref{diffL}, and with the central limit theorem satisfied by $\frac 1 {\sqrt n} \,  \big ( \partial_{\theta_j} L_n(\theta^*_m) \big )_{j\in m} $ provided in Corollary \ref{corol}, this achieves the proof.
\end{proof}
\noindent Now, before establishing Proposition \ref{2+}, three technical lemmas can be stated:
\begin{lemma}\label{lem:nv2}
Under Assumptions {\bf A0}-{\bf A5}, with $8/3<r'\leq r/3$ and $r'<2(\delta-1)$ where $\delta>7/2$ is given in Assumption {\bf A5}, for any $m\in {\cal M}$, there exists $C>0$ such as for any $n\in \N^*$
	\begin{equation} \label{L4}
	\Big\| \Big (\frac{1}{\sqrt n}\, \partial_{\theta_j} L_n ( \theta^*_m) \Big )_{j\in m}\Big\|_{r'} \leq C.
	\end{equation}
\end{lemma}
\begin{proof}
First, for any $m\in {\cal M}$ and $n \in \N^*$,
\begin{eqnarray}
\nonumber \big\| \big ( \partial_{\theta_j} L_n ( \theta^*_m)\big )_{j\in m} \big\|^{r'} &\leq & |m|^{r'/2-1} \, \sum_{j\in m}  \big |
\partial_{\theta_j} L_n ( \theta^*_m) \big |^{r'} \\
\label{inegg}&\leq & \frac {|m|^{r'/2-1}}{2^{r'}} \, \sum_{j\in m} \Big |\sum_{t=1}^n \partial_{\theta_j}\gamma(\theta^*_m, X_t)\Big |^{r'}.
\end{eqnarray}
Now, for all $j \in m$, $\big (\partial_{\theta_j}\gamma(\theta^*_m, X_t)\big)_{t\in \Z}$ is a centered (from the proof of Corollary \ref{corol}) stationary  $\tau_{\phi^{(j)}_\theta}^{(p)}$-weakly dependent where its coefficients $\big (\tau_{\phi^{(j)}_\theta}^{(1)}(s)\big )_s$ satisfies   \eqref{tauphi} (see Lemma \ref{lem00}). Moreover, from the proof of Proposition \ref{theo0}, $\tau_{\phi^{(j)}_\theta}^{(1)}(s)= O\big (s^{1-\delta} \log s\big )$. \\
In Proposition 5.5 of \cite{DDLLLP}, since $\E \big [\big |\partial_{\theta_j}\gamma(\theta^*_m, X_0)\big|^{r'}\big ]<\infty$ from Lemma \ref{lem:7}, it has been established that:
\begin{multline*}
\E \Big [ \Big |\sum_{t=1}^n \partial_{\theta_j}\gamma(\theta^*_m, X_t)\Big |^{r'} \Big ] \leq C_{r'} \, \big (M_{{r'},n}+M_{2,n}^{r'/2} \big )\\ \quad\mbox{where}\quad M_{m,n}:=2n\, \sum_{i=0}^{n-1} (i+1)^{m-2}\,  \tau_{\phi^{(j)}_\theta}^{(1)}(i).
\end{multline*}
Using $8/3<r'< 2(\delta-1)$ with $\delta>7/2$, we obtain that 
$$
M_{{r'},n} \leq C \, n  \sum_{i=1}^n i^{r'-1-\delta} \log(i)\leq C'\, n^{1+r'-\delta} \log(n)=O\big (n^{r'/2} \big)
$$ 
and $M_{2,n} \leq C \,n \sum_{i=1}^n i^{1-\delta} \log(i)\leq C''\, n$. As a consequence, there exists $C>0$ such as for any $n \in \N^*$, 
\begin{equation}\label{prieur}
\E \Big [ \Big |\sum_{t=1}^n \partial_{\theta_j}\gamma(\theta^*_m, X_t)\Big |^{r'} \Big ] \leq C \, n^{r'/2}.
\end{equation}
Then, using \eqref{inegg} and \eqref{prieur}, the proof is established. 
\end{proof}
\begin{lemma} \label{lem:nv1}
Under Assumptions {\bf A0}-{\bf A5},
then for any $m \in \mathcal{M}$, there exists $C>0$ such as for any $n \in \N^*$,
	\[
\Big\| \Big (\frac{1}{\sqrt n}\,\partial_{\theta_j} L_n (\widehat \theta_m)\Big )_{j\in m} \Big\|_{r/3} \leq C.
	\]
\end{lemma}
\begin{proof}
First, from the definition of $\widehat \theta_m$, we have $\partial_{\theta_j} \widehat L_n (\widehat \theta_m)=0$ for any $j\in m$. Then,
\begin{eqnarray*}
\Big\| \Big (\frac{1}{\sqrt n}\,\partial_{\theta_j} L_n (\widehat \theta_m)\Big )_{j\in m}\Big\|_{r/3}&=&	\Big\| \Big (\frac{1}{\sqrt n} \big (\partial_{\theta_j} L_n (\widehat \theta_m)-\partial_{\theta_j} \widehat L_n (\widehat \theta_m)\big) \Big )_{j\in m}\Big\|_{r/3}  \\
&\leq&   \frac{|m|^{(r-6)/2r}}{\sqrt n} \, \sum_{j\in m}\Big\| \partial_{\theta_j} L_n (\widehat \theta_m)-\partial_{\theta_j} \widehat L_n (\widehat \theta_m)\Big\|_{r/3} \\
&\leq &  \frac{|m|^{1/2}}{2 \,\sqrt n} \, \sum_{j\in m}\sum_{t=1}^n  \Big\| \partial_{\theta_j}  \gamma(\widehat \theta_m, X_{t})-\partial_{\theta_j} \widehat \gamma(\widehat \theta_m, X_{t})\Big\|_{r/3}.
\end{eqnarray*}
From the proof of Lemma 2 in \cite{bar2019}, there exists $C>0$ such as
	\begin{multline*}
\E \Big [ \sup_{\theta \in \Theta} \Big\| \partial_{\theta_j}  \gamma(\theta, X_{t})-\partial_{\theta_j} \widehat \gamma(\theta, X_{t})\Big\|^{r/3} \Big ] \\
  \le C\, \Big(\sum_{k \ge t} \alpha_k (f_\theta,\Theta)+ \alpha_k (M_\theta,\Theta)+  \alpha_k (\partial  f_\theta,\Theta)+ \alpha_k (\partial  M_\theta,\Theta)\Big)^{r/3}.
	\end{multline*}
Therefore,
\begin{eqnarray*}
\Big\| \Big (\frac{1}{\sqrt n}\,\partial_{\theta_j} L_n (\widehat \theta_m)  \Big )_{j\in m}\Big\|_{r/3} &\le& C\, \frac{|m|^{3/2}}{2 \,\sqrt n} \,  \sum_{t=1}^n \sum_{k \ge t} \big ( \alpha_k (f_\theta,\Theta)+ \alpha_k (M_\theta,\Theta)\\
& & \hspace{4cm} + \alpha_k (\partial  f_\theta,\Theta)+ \alpha_k (\partial  M_\theta,\Theta)\big)\\
&\le &\frac {C'}{\sqrt n} \,  \sum_{t=1}^n \sum_{j \ge t} j^{-\delta} \leq \frac {C''}{\sqrt n} \, \sum_{t=1}^n  t^{1-\delta} \leq C''',
\end{eqnarray*}
with $C'>0$, $C''>0$ and $C'''>0$ and where the last inequality holds since $\delta>7/2$ under Assumption {\bf A5}.
\end{proof}

\begin{lemma}\label{lem:7}
Under Assumptions {\bf A0}-{\bf A5},	for any $m\in {\cal M}$ and any $\theta \in \Theta_m$,
	\begin{equation}\label{eq:in}
	\Big \|\big (\partial_{\theta_j}  \gamma(\theta, X_{0}) \big )_{j\in m}\Big \|_{r/3} < \infty.
	\end{equation}
\end{lemma}
\begin{proof}
For $j \in m$, we have for any $\theta \in \Theta_m$,
$$
\partial_{\theta_j}  \gamma(\theta, X_{0})=-2\, (M_{\theta}^0)^{-2}(X_0-f_{\theta}^0)\,\partial_{\theta_j} f_{\theta}^0-2\, (M_{\theta}^0)^{-3}(X_0-f_{\theta}^0)^2\partial_{\theta_j}  M_{\theta}^0+2\,(M_{\theta}^0)^{-1}\partial_{\theta_j}  M_{\theta}^0.
$$
Therefore, with Assumption {\bf A3} and Minkowski Inequality,
	\begin{multline*}
	\big \|\big (\partial_{\theta_j}  \gamma(\theta, X_{0}) \big )_{j\in m}\big \|_{r/3}
	\le \frac {2 }{\underline h^{3/2}}\, \Big (  \underline h^{1/2}\, \big \| \big(\partial_{\theta_j} f_{\theta}^0\big), \big (X_0-f_{\theta}^0\big) \big \|_{r/3} \\
	+ \big \| \big(\partial_{\theta_j} M_{\theta}^0\big) \, \big |X_0-f_{\theta}^0\big|^{2} \big \|_{r/3}+ \underline h\, \big \|  \big(\partial_{\theta_p}  M_{\theta}^0\big) \big \|_{r/3} \Big ).
	\end{multline*}
Now, applying the Hölder Inequality, we obtain that there exists $C>0$ such that for any $\theta\in \Theta_m$,
\begin{multline}\label{momentr}
\E\Big[\big| \partial_{\theta_p}  \gamma(\theta^*, X_{0}) \big|^{r/3}\Big] \le C \, \Big (   \big \| \partial_{\theta_j} f_{\theta}^0\big \|_{2r/3}\, \big \|X_0-f_{\theta}^0 \big \|_{2r/3} \\
	+ \big \| \partial_{\theta_j} M_{\theta}^0\big \|_{r} \, \big \|X_0-f_{\theta}^0 \big \|^2_{r}+ \big \| \partial_{\theta_p}  M_{\theta}^0 \big \|_{r/3} \Big ).
\end{multline}
Using Assumption {\bf A0} and {\bf A5} and the proof of Lemma 1 in \cite{barW}, all the right side terms in \eqref{momentr} are finite for any $\theta \in \Theta_m$ and this achieves the proof.
\end{proof}
\noindent Then Proposition \ref{2+} can be established:
\begin{proof}[Proof of Proposition \ref{2+}]
From \eqref{antheta} and \eqref{L2F} with $F(\theta^*_m)$ the positive definite matrix defined in \eqref{matrixF}, we know that for $n$ large enough,
\begin{multline}\label{antheta2}
\big \| \sqrt n \, \big ( (\widehat \theta_m)_i-(\theta^*_m)_i \big )_{i \in m}  \big \|_{r'} \\
=\Big \| \Big (\big ( \frac 1 n \, \partial^2_{\theta_i \theta_j} L_n(\overline \theta_m) \big )_{i,j\in m}\Big )^{-1} \times \frac 1 {\sqrt n} \,  \big ( \partial_{\theta_j} L_n(\widehat \theta_m)-\partial_{\theta_j} L_n(\theta^*_m)\big )_{j\in m} \Big \|_{r'}.
\end{multline}
Therefore, using Hölder and Minkowski inequalities, we obtain:
\begin{eqnarray*}
\big \| \sqrt n \, \big ( (\widehat \theta_m)_i-(\theta^*_m)_i \big )_{i \in m}  \big \|_{r'} & \leq & 
\Big \| \Big (\big ( \frac 1 n \, \partial^2_{\theta_i \theta_j} L_n(\overline \theta_m) \big )_{i,j\in m}\Big )^{-1}\Big \|_{\frac {rr'}{r-3r'} } \\
&& \hspace{2cm} \times \Big \|\frac 1 {\sqrt n} \,  \big ( \partial_{\theta_j} L_n(\widehat \theta_m)-\partial_{\theta_j} L_n(\theta^*_m)\big )_{j\in m} \Big \|_{\frac r 3} \\
& \leq & \Big \| \Big (\big ( \frac 1 n \, \partial^2_{\theta_i \theta_j} L_n(\overline \theta_m) \big )_{i,j\in m}\Big )^{-1}\Big \|_{\frac {rr'}{r-3r'} } \\
&& \hspace{0.1cm} \times \Big (\Big \|\frac 1 {\sqrt n} \,  \big ( \partial_{\theta_j} L_n(\widehat \theta_m)\big )_{j\in m}\Big \|_{\frac r 3} + \Big \|\frac 1 {\sqrt n} \,  \big ( \partial_{\theta_j} L_n(\theta^*_m)\big )_{j\in m} \Big \|_{\frac r 3} \Big ).
\end{eqnarray*}
Now using Assumption {\bf A4}, Lemmas \ref{lem:nv2} and \ref{lem:nv1}, we deduce \eqref{ineg2+}. 
\end{proof}
\subsection{Proofs of Section \ref{sec:resul}}
\begin{proof}[Proof of Lemma \ref{prop1}]
{\bf 1.} From the assumptions, the function $R:~\theta \in \Theta \mapsto R(\theta)$ is a ${\cal C}^2(\Theta)$ function and the Hessian matrix $\partial _{\theta^2} R=-2 \, F$ is a definite positive matrix (see \eqref{matrixF}). Therefore, from a Taylor-Lagrange expansion:
\begin{eqnarray}
\nonumber n \big ( R(\widehat{\theta}_{m})-R(\theta^*_m) \big ) &=& n \Big ( R(\theta^*_m)+\big (\widehat{\theta}_m -\theta^*_m\big )^\top \,\partial _{\theta} R(\theta^*_m) \\
\nonumber&& \hspace{3cm} +\frac{1}{2} \, \big(\widehat{\theta}_m-\theta^*_m\big)^\top\,\partial _{\theta^2} R(\overline \theta) \, \big(\widehat{\theta}_m-\theta^*_m\big)-R(\theta^*_m) \Big ) \\
\label{taylor}&= & \frac{1}{2}\, \big(\sqrt n (\widehat{\theta}_m-\theta^*_m)\big)^\top\,\partial _{\theta^2} R(\overline \theta) \, \big(\sqrt n (\widehat{\theta}_m-\theta^*_m)\big),
\end{eqnarray} 
with $\overline \theta= { \theta^*_m } + c\, \big (\widehat{\theta}_m-\theta^*_m\big )\in \Theta_m$ since $c\in [0,1]$. Using Lemma 4 of \cite{barW} and continuous mapping Theorem, we deduce that:
\begin{equation}
\partial _{\theta^2} R(\overline \theta)=-2 \, F(\overline \theta)\limiteproban -2 \, F(\theta^*_m)\quad \mbox{and}\quad G(\overline \theta)\limiteproban G(\theta^*_m).
\label{eq:oli1}
\end{equation}
Moreover, using the asymptotic normality of $\widehat{\theta}_m$ established in \cite{barW}  and \cite{bar2019}, we have:
\begin{equation}
\sqrt n ((\widehat{\theta}_m)_i-(\theta^*_m)_i)_{i\in m} \limiteloin {\cal N} \Big ( 0 \, , \ \big (F_m(\theta^*_m)\big )^{-1} G_m(\theta^*_m)\big (F_m(\theta^*_m)\big )^{-1} \Big ).
\label{clt1}
\end{equation}
As a consequence, with $Z_n= \big (G_m(\overline \theta)\big )^{-1/2} \,F_m(\overline \theta) \sqrt n ((\widehat{\theta}_m)_i-(\theta^*_m)_i)_{i\in m}\limiteloin {\cal N }(0\, , \, I_{|m|})$ and from \eqref{taylor}, we have
\begin{eqnarray*}
n \big ( R(\widehat{\theta}_{m})-R(\theta^*_m) \big ) &=&- Z_n^\top\,\big (G_m(\overline \theta)\big )^{1/2} \big (F_m(\overline \theta)\big )^{-1}  F_m(\overline \theta) \big (F_m(\overline \theta)\big )^{-1}\big (G_m(\overline \theta)\big )^{1/2}\, Z_n \\
 &=&- Z_n^\top\,\big (G_m(\overline \theta)\big )^{1/2}  \big (F_m(\overline \theta)\big )^{-1} \big (G_m(\overline \theta)\big )^{1/2}\, Z_n.
\end{eqnarray*} 
Define $\displaystyle U^*(m):=- Z^\top\,\big (G_m(\theta^*_m)\big )^{1/2}  \big (F_m(\theta^*_m)\big )^{-1} \big (G_m(\theta^*_m)\big )^{1/2}\, Z$ where $Z \egaleloi {\cal N }(0\, , \,  I_{|m|})$. Then using \eqref{eq:oli1} we obtain
$$
n \big ( R(\widehat{\theta}_{m})-R(\theta^*_m)   \big ) \limiteloin U^*(m).
$$
The computation of the expectation of $U^*_m$ follows from 
\begin{eqnarray*} 
\E \big [ U^*_m\big ] & =& \E \big [\mbox{Trace}\big(U^*_m \big ) \big ]=-\mbox{Trace}\Big (\big (G_m(\theta^*_m)\big )^{1/2}  \big (F_m(\theta^*_m)\big )^{-1} \big (G_m(\theta^*_m)\big )^{1/2} \Big )\\
&= &- \mbox{Trace}\Big( \big (F_m(\theta^*_m)\big )^{-1}  \,G_m(\theta^*_m)\Big ). 
\end{eqnarray*} 
Finally, for establishing $\E \big [n  \big (R(\widehat{\theta}_{m})-R(\theta^*_m)   \big ) \big ] \limiten \E \big [ U^*_m\big ] $, we have to prove that there exists $n_0\in \N$ such as
\begin{equation}\label{espoir}
\sup_{n \geq n_0} \E \big [n \big | R(\widehat{\theta}_{m})-R(\theta^*_m)   \big | \big ]<\infty.
\end{equation}
Indeed, from \eqref{taylor}, we have:
\begin{multline}
n \, \big | R(\widehat{\theta}_{m})-R(\theta^*_m)   \big |\leq  \frac{1}{2}\, \sup_{\theta \in \Theta} \big \| \partial^2 _{\theta^2} R(\theta) \big \|  \,  \big \| \sqrt n (\widehat{\theta}_m-\theta^*_m)\big \| ^2   \\
\Longrightarrow \quad \E \big [n \big | R(\widehat{\theta}_{m})-R(\theta^*_m)  \big | \big ] \leq \frac{\lambda_{\max}}{2}\, \E \big [  \big \| \sqrt n (\widehat{\theta}_m-\theta^*_m)\big \| ^2 \big ],
\label{ineg1}
\end{multline}
since there exists $\lambda_{\max}<\infty$ such as $\big \| \partial _{\theta^2} R(\theta) \big \| \leq \lambda_{\max}$ for any $\theta \in \Theta$ from Assumption {\bf A5} where $\Theta$ is a compact set.\\
Using Proposition \ref{2+}, we know that
\begin{equation*}\label{theta2}
\sup_{n\in \N^*} \big \|  \sqrt n \, \big (\widehat \theta _m-\theta^*_m\big ) \big \|_2 <\infty.
\end{equation*}
Finally using \eqref{ineg1}, we deduce \eqref{espoir}.\\
~\\
{\bf 2.} As in the proof of {\bf 1.}, we use a Taylor-Lagrange expansion of $\widehat \gamma_n(\theta^*_m)$ around $\widehat \theta_m$ since $\partial_\theta \widehat \gamma_n(\widehat \theta_m)=0$. Then, 
\begin{equation*}
n\, \big(\widehat \gamma_n(\theta^*_m)-\widehat \gamma_n(\widehat{\theta}_m) \big)=\frac 1 2 \, \sqrt n (\widehat{\theta}_{m}-\theta^*_m)^\top\, \big( \partial^2_{\theta^2} \widehat \gamma_n(\overline \theta_m) \big )\, \sqrt n\big(\widehat{\theta}_{m}-\theta^*_m\big).
\end{equation*}
But using $\widehat \theta_m \limiteasn \theta^*_m$ and $\displaystyle \E \Big [ \Big \| \frac 1 n \, L_n(\theta)- \frac 1 n \, \widehat L_n(\theta) \Big \|_\Theta \Big ]\limiten 0$, we have 
$$
\big( \partial^2_{\theta^2} \widehat \gamma_n(\overline \theta_m) \big ) \limiteproban -2 \, F(\theta^*_m).
$$ 
Therefore, using the same reasoning as in {\bf 1.}, we deduce that 
$$
n\, \big ( \widehat \gamma_n(\theta^*_m)-\widehat \gamma_n(\widehat{\theta}_m) \big ) \limiteloin U^*(m).
$$
With H\"older Inequality and using $8/3<r'$ defined in Proposition \ref{2+}, we obtain for $n$ large enough
\begin{eqnarray}
\nonumber \E \big [n\, \big ( \widehat \gamma_n(\theta^*_m)-\widehat \gamma_n(\widehat{\theta}_m) \big )\big ] &\le& \big \|\big( \partial^2_{\theta^2} \widehat \gamma_n(\overline \theta_m) \big )^{1/2}\,  \sqrt n (\widehat{\theta}_{m}-\theta^*_m) \big \|^2_2  \\
\label{gamma} &\leq &  \Big \| \big( \partial^2_{\theta^2} \widehat \gamma_n(\overline \theta_m) \big )\Big \|_{\frac {r'} {r'-2}} \,\big \| \sqrt n (\widehat{\theta}_{m}-\theta^*_m) \big \|^2_{r'}.
\end{eqnarray}
Finally, with Proposition \ref{2+} and Lemma 4 of \cite{barW}, we have $\sup_{n\in \N^*} \E \big [\big \| \partial^2_{\theta^2} \widehat \gamma_n(\overline \theta_m) \big \|^4_\Theta \big ]<\infty$ and $\frac {r'} {r'-2}\leq 4$ since $r'>8/3$, and therefore
\begin{equation*}
\sup_{n\in \N^*} \E \big [n\, \big ( \widehat \gamma_n(\theta^*_m)-\widehat \gamma_n(\widehat{\theta}_m) \big )\big ] <\infty,
\end{equation*}
which concludes the proof.
\end{proof}

\begin{proof}[Proof of Proposition  \ref{prop0}] For $m^* \subset m$ we have $\theta^*_m=\theta^*$ and therefore $\ell(\widehat{\theta}_{m},\theta^*)=R(\widehat{\theta}_{m})-R(\theta^*)$ and using \eqref{eq:pr1}, we have $$\displaystyle \E \big [ \ell(\widehat{\theta}_{m},\theta^*) \big ]=\frac 1 n \ \E [I_1(m)] \simn-\frac 1 n \  \mbox{Trace}\Big( \big (F_m(\theta^*_m)\big )^{-1}  \,G_m(\theta^*_m)\Big ).
$$
But the matrix $G(\theta^*_m)$ and $-F(\theta^*_m) $ are positive definite function from Assumption {\bf A2}. Thus if $m^* \subset m$ and $m \neq m^*$, then  $$- \mbox{Trace}\Big( \big (F_{m^*}(\theta^*)\big )^{-1}  \,G_{m^*}(\theta^*)\Big ) < - \mbox{Trace}\Big( \big (F_m(\theta^*)\big )^{-1}  \,G_m(\theta^*)\Big )$$ 
since all the eigenvalues of $-\big (F(\theta^*)\big )^{-1}  \,G(\theta^*)$ are positive. This  implies $\E \big [ \ell(\widehat{\theta}_{m^*},\theta^*) \big ] < \E \big [ \ell(\widehat{\theta}_{m},\theta^*) \big ]$ for $n$ large enough. \\
If $m^* \not \subset m$, then:
$$
\ell(\widehat{\theta}_{m},\theta^*)=\big (R(\widehat{\theta}_{m})-R(\theta^*_m)\big )+ \big (R(\theta^*_m)-R(\theta^*)\big).
$$
Using \eqref{eq:pr1}, we also have 
$$\displaystyle \E \big [ \ell(\widehat{\theta}_{m},\theta^*_m) \big ]=\frac 1 n \ \E [I_1(m)] \simn-\frac 1 n \  \mbox{Trace}\Big( \big (F_m(\theta^*_m)\big )^{-1}  \,G_m(\theta^*_m)\Big ).
$$
But as it was established in \cite{bar2019} that  $\big (R(\theta^*_m)-R(\theta^*)\big)=2 \, DK_L(\theta^*\|\theta_m^*)>0$ since $m \not \subset m^*$. Therefore,  $\E \big [ \ell(\widehat{\theta}_{m^*},\theta^*) \big ] = o\big (\E \big [ \ell(\widehat{\theta}_{m},\theta^*) \big ]\big )$ for any $m$ such as $m^* \not \subset m$. 

\end{proof}

\begin{proof}[Proof of Proposition  \ref{prop}]
The proof of this proposition can be deduced from
\begin{equation} \label{I3}
\E \big [n \,  I_3(m) \big ]=\E \big [n \, \big (R(\theta^*_m) -\widehat \gamma_n(\theta^*_m) \big )  \big ]=v_n^*
\end{equation}
for any $m\in {\cal M}$. 
For establishing \eqref{I3}, we begin by
\begin{equation}\label{decompI3}
I_3(m)=\big (R(\theta^*_m) -\gamma_n(\theta^*_m) \big )+ \big (\gamma_n(\theta^*_m) -\widehat \gamma_n(\theta^*_m)  \big ) :=I_{31}(m)+I_{32}(m).
\end{equation}
Firstly, since $\E \big [\gamma(\theta^*_m,X_0)\big ]=R (\theta^*_m)$ and $(X_t)_{t\in \Z}$ is a stationary times series, then for any $n \in \N^*$,
\begin{equation}\label{I31}
\E \big [\gamma_n(\theta^*_m) \big ]=\frac 1 n \, \sum_{t=1}^n \E \big [\gamma(\theta^*_m,X_t)\big ]=R (\theta^*_m)\quad \Longrightarrow \quad \E \big [I_{31}(m)\big ]=0.
\end{equation}
Secondly, from Assumption {\bf A0} and \cite{barW}, there exists $C>0$ such that for any $t\geq 1$
$$
\E \big [ \big \|\gamma(\theta,X_t)-\widehat \gamma (\theta,X_t)\big \|_\Theta \big ] \leq C \, \sum_{s \geq t} \big (\alpha_s(f_{\theta},\Theta)+ \alpha_s(M_{\theta},\Theta) \big ).
$$
Therefore, there exist $C>0$ and $C'>0$ such that for any $m \in {\cal M}$,
\begin{eqnarray}
\nonumber
\E \big [ \big \|\gamma_n(\theta^*_m)-\widehat \gamma_n(\theta^*_m) \big \|_\Theta \big ] &\leq & \frac C n \,\sum_{t=1}^n  \sum_{s \geq t} \big (\alpha_s(f_{\theta},\Theta)+ \alpha_s(M_{\theta},\Theta) \big ) \\
\label{borneC} &\leq & \frac C n \,\sum_{t=1}^n t^{1-\delta} \leq \frac{C'} n,
\end{eqnarray}
since $\delta>7/2$ from Assumption {\bf A5}.
Moreover, for any $m \in {\cal M}$ such as $m^*\subset m$ (overfitting setting), we have $\gamma_n(\theta^*_m)-\widehat \gamma_n(\theta^*_m)=\gamma_n(\theta^*_{m^*})-\widehat \gamma_n(\theta^*_{m^*})$. Using this and \eqref{borneC} we deduce that for any $m\in {\cal M}$, there exists a bounded sequence $(v^*_n)_{n\in \N^*}$ not depending on $m$ when $m^* \subset m$ satisfying
\begin{equation}\label{EI32}
\E \big [I_{32}(m) \big ] = \frac {v^*_n}{n}.
\end{equation}
Using also Lemma \ref{prop1}, this implies the asymptotic behavior of $\E\big[ \mbox{pen}_{id}(m)\big] $.
\end{proof}
\noindent Now we establish a preliminary lemma that is an important step towards the proof of Theorem \ref{theo:1}.
\begin{lemma}\label{lem:deb}
Let $\mbox{pen}: m\in \mathcal{M}_n \mapsto \textnormal{pen}(m)\in \R_+$. Then,
	\begin{equation}\label{eq:lo}
	\ell(\widehat{\theta}_{\widehat{m}_{\textnormal{pen}}},\theta^* ) \leq \underset{m \in \mathcal{M}}{\min} \big\{\ell(\widehat{\theta}_{m},\theta^*)\big\}+ \big (\textnormal{pen}(\widehat m _{id})-\textnormal{pen}(\widehat m _{\textnormal{pen}}) \big )-\big (\textnormal{pen}_{id}(\widehat m _{id})-\textnormal{pen}_{id}(\widehat m _{\textnormal{pen}}) \big ) .
	\end{equation}
\end{lemma}

\begin{proof}
By definition, for any $m\in {\cal M}$, 
\begin{equation}\label{Cl}
\widehat C_{\textnormal{pen}_{id}}(m)=R(\widehat \theta_m)=\ell(\widehat{\theta}_{m},\theta^* )+R(\theta^*).
\end{equation} 
As a consequence,
\begin{equation}\label{minid}
\underset{m \in \mathcal{M}}{\min} \big\{\ell(\widehat{\theta}_{m},\theta^*)\big\}=\ell(\widehat{\theta}_{\widehat m_{id}},\theta^*)=\underset{m \in \mathcal{M}}{\min} \big \{\widehat C_{\textnormal{pen}_{id}}(m)\big\}-R(\theta^*).
\end{equation}
For any $m \in {\cal M}$, we also have
$$
\widehat C_{\textnormal{pen}}(m)=\widehat C_{\textnormal{pen}_{id}}(m)+\textnormal{pen}
(m)-\textnormal{pen}_{id}(m).
$$
By definition of $\widehat m_{\textnormal{pen}}$, we have $\widehat C_{\textnormal{pen}}(\widehat m_{\textnormal{pen}}) \leq \widehat C_{\textnormal{pen}}(\widehat m_{id})$. Therefore,
\begin{eqnarray*}
\widehat C_{\textnormal{pen}}(\widehat m_{\textnormal{pen}}) & \leq & \widehat C_{\textnormal{pen}_{id}}(\widehat m_{{id}})+\textnormal{pen}(\widehat m_{{id}})-\textnormal{pen}_{id}(\widehat m_{{id}}) \\
\widehat C_{\textnormal{pen}_{id}}(\widehat m_{\textnormal{pen}})+\textnormal{pen}(\widehat m_{\textnormal{pen}})-\textnormal{pen}_{id}(\widehat m_{\textnormal{pen}}) 
& \leq & \widehat C_{\textnormal{pen}_{id}}(\widehat m_{{id}})+\textnormal{pen}(\widehat m_{{id}})-\textnormal{pen}_{id}(\widehat m_{{id}}).
\end{eqnarray*}
By replacing $\widehat C_{\textnormal{pen}_{id}}(m)$ by $\ell(\widehat \theta_m,\theta^*)+R(\theta^*)$ following \eqref{Cl} and using \eqref{minid}, then \eqref{eq:lo} is established.
\end{proof}
\begin{proof}[Proof of Theorem \ref{theo:0}]
Let ${\cal M}^*=\big \{ m\in {\cal M}, ~m^*\subset m \big \}$ and ${\cal M}'={\cal M} \setminus {\cal M}^*$.  Let $m \in {\cal M}'$. We have:
\begin{eqnarray*}
\P \big( \widehat m _{{\textnormal{pen}}}=m\big) & \leq & \P \big ( \widehat C_{{\textnormal{pen}}}(m)\leq \widehat C_{{\textnormal{pen}}}(m^*) \big ) \\
\nonumber  & \leq & \P \Big \{ \widehat \gamma_n(\widehat \theta_m)-\widehat \gamma_n(\widehat \theta_{m^*}) \leq {\textnormal{pen}}(m^*)-  {\textnormal{pen}}(m)\Big \} \\
& \leq &  \P \Big \{    n \, \big (\widehat \gamma_n(\widehat \theta_m)-\widehat \gamma_n(\theta_m^*)\big )+n \, \big (\widehat \gamma_n(\theta_m^*)-R(\theta_m^*)\big )+n \, \big (R(\theta^*)-\widehat \gamma_n(\theta^*)\big ) \\
\nonumber  && \hspace{0.5cm}  + n \, \big (\widehat \gamma_n(\theta^*)-\widehat \gamma_n(\widehat \theta_{m^*})\big )  \leq n \, \big (R(\theta^*)-R(\theta_m^*) \big )+n \, \big ( {\textnormal{pen}}(m^*)-  {\textnormal{pen}}(m)\big )\Big \} \\
& \leq & \P \Big \{  Z_1+Z_2+Z_3+Z_4+Z5\leq -2 \,n \,DK_L(\theta^*\|\theta_m^*)\Big \}
\end{eqnarray*}
with $Z_5=n \, \big ( {\textnormal{pen}}(m)-  {\textnormal{pen}}(m^*)\big )$ and with $R(\theta^*)-R(\theta_m^*)=-2 \, DK_L(\theta^*\|\theta_m^*)<0$ since $m \not \subset m^*$ from \cite{bar2019}. Now, using $\P(Z_1+\cdots+Z_5\leq c)\leq \P(Z_1\leq c/5)+ \cdots +\P(Z_5\leq c/5)$ for any random variables $Z_i$ and real number $c$, we obtain:
\begin{equation}
\P \big( \widehat m _{{\textnormal{pen}}}=m\big) \leq \sum_{i=1}^5 \P \big ( Z_i \leq c_n \big ),
\end{equation}
where $c_n=-\frac 2 5 \,n \,DK_L(\theta^*\|\theta_m^*)$. \\
Let $Z_1:=n \, \big (\widehat \gamma_n(\widehat \theta_m)-\widehat \gamma_n(\theta_m^*)\big )$. Following the same computations than in \eqref{gamma}, with $8/3< r'\leq r/3$ and $r'< 2 (\delta-1)$ defined in Proposition \ref{2+}, and Hölder Inequality,
\begin{eqnarray*}
\E \big [\big | Z_1 \big |^{\frac {3r'}8} \big ]&\leq &\big \|\big( \partial^2_{\theta^2} \widehat \gamma_n(\overline \theta_m) \big )^{1/2}\,  \sqrt n (\widehat{\theta}_{m}-\theta^*_m) \big \|^{\frac {3r'}4}_{\frac {3r'}4} \\
&\leq &  \Big \| \big( \partial^2_{\theta^2} \widehat \gamma_n(\overline \theta_m) \big )\Big \|_{\frac {3r'} {2}} \,\big \| \sqrt n (\widehat{\theta}_{m}-\theta^*_m) \big \|^{\frac {3r'}4}_{r'}.
\end{eqnarray*}
Therefore, using Proposition \ref{2+},
\begin{multline}\label{Z1}
\P \big ( Z_1\leq c_n\big )\leq \P \Big ( |Z_1|^{\frac {3r'}8}\geq \big (|c_n|\big )^{\frac {3r'}8}\Big )
\leq  \E \big [\big | Z_1 \big |^{\frac {3r'}8} \big ]\frac 1 {|c_n|^{\frac {3r'}8}} \\
\Longrightarrow \quad \P \big ( Z_1\leq c_n\big )=O\Big (\frac 1 {n^{\frac {3r'}8} } \Big )=o\big (\frac 1 n\big ),
\end{multline}
since $3r'/8>1$. The same kind of computations can also be done for $Z_4:=n \, \big (\widehat \gamma_n(\theta^*)-\widehat \gamma_n(\widehat \theta_{m^*})\big )$ and we also obtain $\P \big ( Z_4\leq c_n\big )=o\big (\frac 1 n\big )$. \\
~\\
Consider now $Z_2:=n \, \big (\widehat \gamma_n(\theta_m^*)-R(\theta_m^*)\big )$. Then, 
\begin{equation*}
\E \big [\big | Z_2 \big |^{8/3} \big ]\leq  2^{5/3} \,  \Big (\E \big [\big \| \widehat L_n(\theta)-  L_n(\theta) \big \|^{8/3} _\Theta \big ] + n^{8/3}\,  \E \big [\big |\sum_{k=1}^n \big (\gamma(\theta_m^*,X_k)-R(\theta_m^*) \big )\big |^{8/3} \big ] \Big ).
\end{equation*}
Using \cite{barW}, we know that $\sup_{n\in \N^*} E \big [\big \| \widehat L_n(\theta)-  L_n(\theta) \big \|^{8/3} _\Theta \big ] <\infty$ from Assumption {\bf A5} and since $\delta>7/2>2$. Now, consider $Y_k:=\gamma(\theta_m^*,X_k)-R(\theta_m^*)$. Then, $(Y_k)_{k\in \Z}$ is a stationary time series, $\tau_Y$-weakly dependent because, using the same type of arguments as in the proof of Lemma \ref{lem00}, we have:
$$
\tau_Y(s)\leq \sum_{\ell=1}^\infty \big ( \alpha_\ell(f_\theta,\Theta)+\alpha_\ell(M_\theta,\Theta)\big )\, \lambda_{s+1-\ell},
$$
with $\lambda$ defined in Lemma \ref{lem0}. Therefore, using Assumption {\bf A5}, we also have $\tau_Y(s) = O \big ( s^{\delta-1} \log (s) \big )$, with $\delta>7/2$. Now, using the same type of arguments as in the proof of Lemma \ref{lem:nv2},
$$
\E \big [ \big | \sum_{k=1}^n Y_k\big | ^{8/3} \big ] \leq C_{8/3} \, \big (M_{{8/3},n}+M_{2,n}^{4/3} \big ),
$$
and $M_{2,n}\leq C \, n$ while $M_{{8/3},n}\leq C \, n  \sum_{i=1}^n i^{8/3-1-\delta} \log(i)=o\big (n^{4/3} \big)$. Therefore, there exists $C>0$ such that for any $n \in \N^*$, 
$$
\E \big [ \big | \sum_{k=1}^n Y_k\big | ^{8/3} \big ] \leq C \, n^{4/3}.
$$
Finally, we deduce that there exists $C>0$ such that for any $n \in \N^*$,
\begin{equation}\label{Z21}
\E \big [\big | Z_2 \big |^{8/3} \big ]\leq C \,  n^{4/3}.
\end{equation}
This result and Markov Inequality imply, 
\begin{multline}\label{Z2}
\P \big ( Z_2\leq c_n\big )\leq \P \Big ( |Z_2|^{8/3}\geq \big (|c_n|\big )^{8/3}\Big )
\leq \E \big [\big | Z_2 \big |^{8/3} \big ] \, \frac 1 {|c_n|^{8/3}} \\
\Longrightarrow \quad \P \big ( Z_2\leq c_n\big )=O\Big (n^{4/3} \, \frac 1 {|c_n|^{8/3 }} \Big )=O\Big (\frac 1 {n^{4/3}}\Big ),
\end{multline}
We obtain the same bound for $Z_3:=n \, \big (R(\theta^*)-\widehat \gamma_n(\theta^*)\big )$. \\
Finally using the assumption \eqref{condpen}, we have:
\begin{multline}\label{Z5}
n \, \P \big ( Z_5 \leq c_n \big )=n \,  \P \Big ( \big ({\textnormal{pen}}(m)-  {\textnormal{pen}}(m^*)\big ) \leq - \frac 2 5 \, DK_L(\theta^*\|\theta_m^*) \Big )  \\
\leq n \, \P \Big ( {\textnormal{pen}}(m^*) \geq  \frac 2 5 \, DK_L(\theta^*\|\theta_m^*) \Big )  \limiten 0.
\end{multline}
By this way, \eqref{theo0pen} is established.
\end{proof}

\begin{proof}[Proof of Theorem \ref{theo:1}]
The proof is mainly based on Lemma  \ref{lem:deb}. 
From the proof of Lemma \ref{prop1}, we deduce that for any $m \in {\cal M}$, there exists two positive random variables $Y(m)$ and $Z(m)$ such as $n \, \big |I_1(m)+I_2(m)\big | \leq Y(m)$ and  $n \, \big |I_3(m) \big | \leq Z(m)$ for any $n \in \N^*$. Moreover, $Y(m)$ and $Z(m)$ have bounded expectations. Therefore, using Markov Inequality, since ${\cal M}$ is supposed to be a finite family of models, 
for any $\varepsilon>0$ there exists $K'_\varepsilon>0$ such as 
\begin{equation*}\label{eq:w}
\limsup _{ n \to \infty} \max _{m \in {\cal M}} \,  \P \Big (n \, \textnormal{pen}_{id}(m)  \geq K'_\varepsilon  \Big )\leq \varepsilon.
\end{equation*}
Therefore, using this inequality and \eqref{eq:penal}, we deduce that for any $\varepsilon>0$ there exist $M_\varepsilon>0$ and $N_\varepsilon \in \N^*$ such that for any $n \geq N_\varepsilon$,
\begin{equation}\label{eq:al}
\P \Big (  n \, \big | \big (\textnormal{pen}(\widehat m _{id})-\textnormal{pen}(\widehat m _{\textnormal{pen}}) \big )-\big (\textnormal{pen}_{id}(\widehat m _{id})-\textnormal{pen}_{id}(\widehat m _{\textnormal{pen}}) \big ) \big | \leq M_\varepsilon \Big )
\geq 1 -\varepsilon.
\end{equation}
The proof of \eqref{bornepen} is now completed from \eqref{eq:lo} of Lemma \ref{lem:deb} and \eqref{eq:al}. 
\end{proof}
\begin{proof}[Proof of Theorem \ref{Theo7}]
Using the same tricks than in Lemma \ref{lem:deb}, we obtain:
\begin{multline}\label{lemmetilde}
\ell(\widehat{\theta}_{\widehat{m}_{{\textnormal{pen}}}},\theta^* )  \leq \ell(\widehat{\theta}_{m^*},\theta^*)+
 \big ( {\textnormal{pen}}(m^*)-{\textnormal{pen}}(\widehat m _{{\textnormal{pen}}}) \big ) - \big ( \textnormal{pen}_{id}(m^*)  -\textnormal{pen}_{id}(\widehat m _{ {\textnormal{pen}}}) \big ) .
\end{multline}
Let ${\cal M}^*=\big \{ m\in {\cal M}, ~m^*\subset m \big \}$ and ${\cal M}'={\cal M} \setminus {\cal M}^*$. 
Now, for $m\in {\cal M}^*$ and $m \neq m^*$, as in the beginning of the proof of Theorem \ref{theo:0}, we have:
\begin{eqnarray*}
\P \big (\widehat m _{\textnormal{pen}}=m \big ) & \leq & \P \big ( \widehat C_{\textnormal{pen}}(m)\leq  \widehat C_{\textnormal{pen}}(m^*) \big ) \\
 & \leq &  \P \Big \{    n \, \big (\widehat \gamma_n(\widehat \theta_m)-\widehat \gamma_n(\theta^*) \big )+n \, \big (\widehat \gamma_n(\theta^*)-\widehat \gamma_n(\widehat \theta_{m^*})\big ) \leq n \, \big ( {\textnormal{pen}}(m^*)-  {\textnormal{pen}}(m)\big )\Big \} \\
& \leq &  \P \Big \{   n \, \big (\widehat \gamma_n(\theta^*) -\widehat \gamma_n(\widehat \theta_m)\big ) \geq \E[f_n(m)]\Big \}  + \P \Big \{  n \, \big (\widehat \gamma_n(\widehat \theta_{m^*})-\widehat \gamma_n(\theta^*)\big )  \geq \E[f_n(m)]\Big \}\\
\nonumber  &&\hspace{6cm} + \P \Big \{3\big (f_n(m) -\E[f_n(m)]\big ) \geq \E[f_n(m)]\Big \}
\end{eqnarray*}
with $f_n(m)=\frac n 3 \, e_n(m)$ and $e_n(m)={\textnormal{pen}}(m)-  {\textnormal{pen}}(m^*)>0$ since $m^*\subset m$ and $m\neq m^*$.  \\
Using exactly the same arguments as in the proof of Theorem \ref{theo:0}, there exists $C_{1}>0$ such that for $n$ large enough,
\begin{multline}\label{Z14}
 \P \Big \{   n \, \big (\widehat \gamma_n(\theta^*)-\widehat \gamma_n(\widehat \theta_m) \big ) \geq \E[f_n(m)]\Big \} +\P \Big \{  n \, \big (\widehat \gamma_n(\widehat \theta_{m^*}) -\widehat \gamma_n(\theta^*)\big )  \geq \E[f_n(m)]\Big \} \\ \leq \frac{C_{1}}{\E[f_n(m)]^{\frac{3r'}{8}}}
\end{multline}
where $r'>\frac 8 3$. Moreover, from Markov Inequality we have
\begin{multline}\label{Z22}
 \P \Big \{3\big (f_n(m) -\E[f_n(m)]\big ) \geq \E[f_n(m)]\Big \} \leq \P \Big \{\big |f_n(m)-\E[f_n(m)] \big |\geq \frac 1 3 \, \E[f_n(m)]\Big \} \\
\leq \frac{3 \, \E \big [\big |f_n(m)-\E[f_n(m)] \big | \big ]}{\E[f_n(m)]}.
\end{multline}
As a consequence, from \eqref{Z14} and \eqref{Z22}, with $\kappa>1$, for $m\in {\cal M}^*$ and $m \neq m^*$ and $n$ large enough 
\begin{equation}\label{MajPm}
\P \big (\widehat m _{\textnormal{pen}}=m \big )  \leq \frac{C'_{1}}{(n\,\E[e_n(m)])^\kappa} + 3 \, \frac{n \, \E \big [\big |e_n(m)-\E[e_n(m)]\big | \big ]}{n \,\E[e_n(m)]} \limiten 0.
\end{equation}
Using this result as well as \eqref{MajPm}, one finally obtain \eqref{consist}. \\
Moreover,
\begin{eqnarray*}
&& \E\Big [ \big | {\textnormal{pen}}(m^*)-{\textnormal{pen}}(\widehat m _{{\textnormal{pen}}}) \big |\Big ] \\
&&\hspace{1.5cm} =\sum_{m\in {\cal M}^*} \E\Big [ \big | {\textnormal{pen}}(m^*)-{\textnormal{pen}}(\widehat m _{{\textnormal{pen}}}) \big |~\Big | ~  \widehat m _{{\textnormal{pen}}}=m \Big ] \, \P \Big \{\widehat m _{{\textnormal{pen}}}=m \Big \} \\
&& \hspace{4cm} + \sum_{m\in {\cal M}'} \E\Big [ \big | {\textnormal{pen}}(m^*)-{\textnormal{pen}}(\widehat m _{{\textnormal{pen}}}) \big | ~\Big | ~  \widehat m _{{\textnormal{pen}}}=m \Big ] \, \P \Big \{\widehat m _{{\textnormal{pen}}}=m \Big \} \\
&&\hspace{1.5cm} =\sum_{m\in {\cal M}^*} \E\Big [ \big | {\textnormal{pen}}(m^*)-{\textnormal{pen}}(m) \big |\Big ] \, \P \Big \{\widehat m _{{\textnormal{pen}}}=m \Big \} \\
&& \hspace{4cm} + \sum_{m\in {\cal M}'} \E\Big [ \big | {\textnormal{pen}}(m^*)-{\textnormal{pen}}( m ) \big | \Big ] \, \P \Big \{\widehat m _{{\textnormal{pen}}}=m \Big \} \\
&&\hspace{1.5cm} \leq \frac 1 n \, \sum_{m\in {\cal M}^*} \Big (\frac{C'_{1}}{(n\,\E[e_n(m)])^{\kappa-1} } +  n \, \E \big [\big |e_n(m)-\E[e_n(m)] \big | \big ] \Big )+ \frac {C'_2} n \, \sum_{m\in {\cal M}'} \E[e_n(m)], 
\end{eqnarray*}
from \eqref{MajPm}, where $\kappa>1$ and assumption \eqref{condpen}. As a consequence, using the conditions \eqref{Cond7} of Theorem \ref{Theo7},  
\begin{equation}\label{pen0}
n \, \E\Big [ \big | {\textnormal{pen}}(m^*)-{\textnormal{pen}}(\widehat m _{{\textnormal{pen}}}) \big |\Big ] \limiten 0.
\end{equation}
Moreover, using \eqref{borneC}, there exists $C_3>0$ such as for any $m\in {\cal M}$, 
\begin{equation*}
n \, \E\Big [ \big | \textnormal{pen}_{id}(m^*)  -\textnormal{pen}_{id}(m) \big |\Big ] \leq C_3.
\end{equation*}
Using once again the decomposition on ${\cal M}^*$ and ${\cal M}'$, and $\P\big \{ \widehat m _{{\textnormal{pen}}}=m\big \} \limiten 0$ for $m\neq m^*$, we deduce
\begin{equation}\label{pen0bis}
n \, \E\Big [ \big | \textnormal{pen}_{id}(m^*)  -\textnormal{pen}_{id}( \widehat m _{{\textnormal{pen}}}) \big |\Big ] \limiten 0.
\end{equation}
Using the limit \eqref{pen0bis} as well as \eqref{pen0}, we deduce with Markov inequality that 
$$
n\, \Big [ \big ( {\textnormal{pen}}(m^*)-{\textnormal{pen}}(\widehat m _{{\textnormal{pen}}}) \big ) - \big ( \textnormal{pen}_{id}(m^*)  -\textnormal{pen}_{id}(\widehat m _{ {\textnormal{pen}}}) \big ) \Big ] \limiteproban 0
$$
inducing the proof of \eqref{bornewidepen} from \eqref{lemmetilde}. \\
Now, using the expectation of \eqref{lemmetilde}, we also obtain
\begin{multline*}\label{lemmetildebis}
\E \big [ \ell(\widehat{\theta}_{\widehat{m}_{{\textnormal{pen}}}},\theta^* ) \big ] \leq \E \big [ \ell(\widehat{\theta}_{m^*},\theta^*) \big ]+
\E \big [  \big |{\textnormal{pen}}(m^*)-{\textnormal{pen}}(\widehat m _{{\textnormal{pen}}}) \big | \big ] \\ + \E \big [ \big | \textnormal{pen}_{id}(m^*)  -\textnormal{pen}_{id}(\widehat m _{ {\textnormal{pen}}}) \big | \big ]  .
\end{multline*}
Now, by using \eqref{pen0} and \eqref{pen0bis} as well as Proposition \ref{prop0}, we obtain the proof of \eqref{borneopti}.
\end{proof}
\begin{proof}[Proof of Theorem \ref{Theo8}]
First we will prove that the probability of overfitting is asymptotically positive, which is
\begin{equation}\label{probafaux}
\liminf_{ n \to \infty} \, \P\Big (\widehat m_{\textnormal{pen}} \in {\cal M}^*\setminus\{m^*\} \Big )> 0.
\end{equation} 
Indeed, let $m\in {\cal M}^*\setminus\{m^*\} $. We have:
\begin{multline*}
\P \Big ( \widehat C_{{\textnormal{pen}}}(m)\leq \widehat C_{{\textnormal{pen}}}(m^*) \Big ) \\
 = \P \Big (  n \, \big (\widehat \gamma_n(\widehat \theta_{m^*})-\widehat \gamma_n(\theta^*) \big )+n \, \big (\widehat \gamma_n(\theta^*)-\widehat \gamma_n(\widehat \theta_{m})\big ) \geq g(m)-  g(m^*)\Big ) \\
= \P \Big (  n \,I_2(m) -n \,I_2(m^*) \geq g(m)-  g(m^*)\Big ),
\end{multline*}
using the notations of Lemma \ref{prop1} and since $\textnormal{pen}(m)=g(m)/n$ for any $m\in {\cal M}$. But using Lemma \ref{prop1} and Proposition \ref{prop0}, we know that if $m\in {\cal M}^*\setminus\{m^*\} $ then:
$$
n \,I_2(m)  \limiteloin U^*_m, \quad n \,I_2(m^*)  \limiteloin U^*_{m^*} \quad \mbox{and}\quad \E[U^*_{m^*} ]<\E[U^*_{m} ] 
$$
Moreover, $U^*_m=- Z_{(m)}^\top\,\big (G_m(\theta^*_m)\big )^{1/2}  \big (F_m(\theta^*_m)\big )^{-1} \big (G_m(\theta^*_m)\big )^{1/2}\, Z_{(m)}$ and $Z_{(m)}\egaleloi {\cal N}\big (0, I_{|m|}\big ) $. With a symmetric matrix diagonalization $\big (G_m(\theta^*_m)\big )^{1/2}  \big (F_m(\theta^*_m)\big )^{-1} \big (G_m(\theta^*_m)\big )^{1/2}=-P_{(m)} D_{(m)}P_{(m)}^\top$ where $D_{(m)}$ is a diagonal matrix with positive diagonal components, leading to 
$$
U^*_m=\widetilde {Z}_{(m)}^\top\,D_{(m)}\, \widetilde {Z}_{(m)}\quad\mbox{and}\quad \widetilde {Z}_{(m)}\egaleloi {\cal N}\big (0, I_{|m|}\big ).
$$
Therefore we can write for $m \in {\cal M}^*\setminus \{m^*\}$, $U^*_m=V^*_{m^*}+W^*_{m\setminus m^*}$ with 
$$
W^*_{m\setminus m^*}= \widetilde {Z}_{(m\setminus m^*)}^\top\,D_{(m\setminus m^*)}\, \widetilde {Z}_{(m\setminus m^*)}\quad\mbox{and}\quad \widetilde {Z}_{(m\setminus m^*)}\egaleloi {\cal N}\big (0, I_{|m|-|m^*|}\big )
$$
and $V^*_{m^*}$ and $W^*_{m\setminus m^*}$ are two independent random variables. Moreover, it is clear  that $W^*_{m\setminus m^*}$ behaves as a weighted $\chi^2(|m|-|m^*|)$ random variable, and therefore for any $c>0$, $\P \big (W^*_{m\setminus m^*}>c \big )>0$. Using $n \,I_2(m)-n \,I_2(m^*)  \limiteloin (V^*_{m^*}-U^*_{m^*})+W^*_{m\setminus m^*}$ with $(V^*_{m^*}-U^*_{m^*})$ and $W^*_{m\setminus m^*}$ independent random variables, we deduce that 
\begin{equation}\label{proba**}
\P \Big (  n \,I_2(m) -n \,I_2(m^*) \geq g(m)-  g(m^*)\Big ) \limiten p_{m,m^*}>0,
\end{equation}
and this proves \eqref{probafaux}. \\
~\\
Now, using previous notations, we have:
\begin{eqnarray*}
\E \big [\ell(\widehat{\theta}_{\widehat{m}_{{\textnormal{pen}}}},\theta^* ) \big ] &= & \sum_{m \in {\cal M}^*} \E \big [\ell(\widehat{\theta}_{m},\theta^* ) \big ] \, \P \big ( \widehat{m}_{{\textnormal{pen}}}=m \big )+\sum_{m \in {\cal M}'} \E \big [\ell(\widehat{\theta}_{m},\theta^* ) \big ] \, \P \big ( \widehat{m}_{{\textnormal{pen}}}=m \big ) \\
&= & \E \big [\ell(\widehat{\theta}_{m^*},\theta^* ) \big ] +\hspace{-5mm}  \sum_{m \in {\cal M}^*\setminus{m^*} }\hspace{-5mm} \big ( \E \big [\ell(\widehat{\theta}_{m},\theta^* ) \big ]-\E \big [\ell(\widehat{\theta}_{m^*},\theta^* ) \big ] \big ) \, \P \big ( \widehat{m}_{{\textnormal{pen}}}=m \big ) \\
&& \hspace{2cm} +\sum_{m \in {\cal M}'} \big ( \E \big [\ell(\widehat{\theta}_{m},\theta^* ) \big ]-\E \big [\ell(\widehat{\theta}_{m^*},\theta^* ) \big ] \big ) \, \P \big ( \widehat{m}_{{\textnormal{pen}}}=m \big ).
\end{eqnarray*}
Now using Proposition \ref{prop0}, we know that for $n$ large enough $\E \big [\ell(\widehat{\theta}_{m},\theta^* ) \big ]-\E \big [\ell(\widehat{\theta}_{m^*},\theta^* ) \big ] \big ) \geq 0$ for any $m\in {\cal M}$. Moreover, for $m\in {\cal M}^*\setminus{m^*}$ and $n$ large enough,
\begin{multline*}
\E \big [\ell(\widehat{\theta}_{m},\theta^* ) \big ]-\E \big [\ell(\widehat{\theta}_{m^*},\theta^* ) \big ] \\ 
\geq \frac 1 {2n} \Big (  \mbox{Trace}\Big( \big (-F_m(\theta^*_m)\big )^{-1}  \,G_m(\theta^*_m)\Big ) - \mbox{Trace}\Big( \big (-F_{m^*}(\theta^*_{m^*})\big )^{-1}  \,G_{m^*}(\theta^*_{m^*})\Big ) \Big ) \\
\geq \frac 1 {n} \, K(m,m^*),
\end{multline*}
where $M(m,m^*)>0$. As a consequence, for $n$ large enough, with $p(m,m^*)$ defined in \eqref{proba**},
\begin{equation*}
\E \big [\ell(\widehat{\theta}_{\widehat{m}_{{\textnormal{pen}}}},\theta^* ) \big ] \geq 
 \E \big [\ell(\widehat{\theta}_{m^*},\theta^* ) \big ] + \frac 1 {2n}\hspace{-2mm} \sum_{m \in {\cal M}^*\setminus{m^*} }\hspace{-5mm}K(m,m^*) \, p(m,m') 
\geq \E \big [\ell(\widehat{\theta}_{m^*},\theta^* ) \big ] + \frac {M} {n},
\end{equation*}
with $\displaystyle M=\frac 1 {2}\hspace{-2mm} \sum_{m \in {\cal M}^*\setminus{m^*} }\hspace{-5mm}K(m,m^*) \, p(m,m')>0$ and this achieves the proof.
\end{proof}
\begin{proof}[Proof of Theorem \ref{theo:bic}] We first verify conditions (C1) and (C2) of \cite{Chen} that are sufficient to imply Conditions (i), (ii) and (iii) of \cite{ktk}. Condition (C1) requires that $\widehat \sigma_n$ the largest eigenvalue of $\big (-\big (\partial ^2_{\theta_i\theta_j} \widehat L(\widehat \theta_m)\big )_{i,j \in m}\big )^{-1}$ satisfies $\widehat \sigma_n \limiteasn 0$, which is satisfied since it was already established that $\frac 1 n \, \big (\partial ^2_{\theta_i\theta_j} \widehat L(\widehat \theta_m)\big )_{i,j \in m} \limiteasn  F_m(\theta^*_m)$ and $F_m(\theta^*_m)$ is a negative definite matrix. Moreover, condition (C2) is also satisfied because $\theta_m \in \Theta_m \mapsto \big (\partial ^2_{\theta_i\theta_j} \widehat L(\theta_m)\big )_{i,j \in m} $ and $\theta_m \in \Theta_m \mapsto \big (\big (\partial ^2_{\theta_i\theta_j} \widehat L(\theta_m)\big )_{i,j \in m}\big )^{-1}$ are continuous functions for $n$ large enough. Therefore, using $h_n=-\frac 1 n \, \widehat L_n$, the assumptions of Theorem 1 of \cite{ktk} are satisfied and this implies that: 
\begin{multline*}
\int_{\Theta_m}b_m(\theta)\,  \exp\big(\widehat{L}_n(\theta)\big)\, d\theta = \exp\big(\widehat{L}_n(\widehat{\theta}_m) \big)\, \big(2\, \pi\big)^{|m|/2} \\
\times \det\Big (n \, \big (-\frac 1 n \, \partial ^2_{\theta_i\theta_j} \widehat L(\widehat \theta_m)\big )_{i,j \in m}\Big)^{-1/2}\Big ( b_m(\widehat \theta_m) +O(n^{-1})\Big )\quad a.s.
\end{multline*}
As a consequence, we have:
\begin{eqnarray*}
\widehat{S}(m,X) &= & -\log (|{\cal M}|)+\log \Big[ \int_{\Theta_m} b_m(\theta)\,  \exp\big( \widehat{L}_n(\theta)\big) \, d\theta   \Big]\\
&= &\widehat{L}_n(\widehat{\theta}_m)-\frac{\log(n)}{2}\, |m|+\log\big (b_m(\widehat \theta_m)\big ) \\
&&\hspace{1cm}+\frac{\log(2\pi)} 2 \, |m|-\frac{1}{2}\, \log \big ( \det\big (-\widehat F_n(m)\big)\big) -\log (|{\cal M}|)+O(n^{-1})\quad a.s.
\end{eqnarray*}
and Theorem  \ref{theo:bic} follows.
\end{proof}

\paragraph*{Aknowledgments}  This work has received funding from the European Union's Horizon 2020 research and innovation programme
under the Marie Sklodowska-Curie grant agreement No 754362.
We also thank Christian Francq for some really important suggestions.
\bibliographystyle{imsart-number} 
\bibliography{biblio10}
\end{document}